\documentclass[10pt]{article}
\usepackage{graphics}
\usepackage{graphicx}

\usepackage[a4paper, left=35mm,right=35mm,top=34mm,bottom=34mm]{geometry}
\usepackage[utf8]{inputenc}
\usepackage[T1]{fontenc}
\usepackage[english]{babel}

\usepackage{enumerate}
\usepackage{graphicx}
\usepackage{hyperref}
\hypersetup{
    colorlinks=true,
    linkcolor=blue,
    filecolor=magenta,      
    urlcolor=cyan,
}
\usepackage{lipsum}
\usepackage{amsfonts}
\usepackage{graphicx}
\usepackage{epstopdf}
\usepackage{algorithmic}
\usepackage{lipsum}
\usepackage{amsfonts}
\usepackage{graphicx}
\usepackage{epstopdf}

\usepackage{hyperref}
\ifpdf
  \DeclareGraphicsExtensions{.eps,.pdf,.png,.jpg}
\else
  \DeclareGraphicsExtensions{.eps}
\fi

\usepackage{xcolor}
\usepackage{url}

\usepackage{float}
\usepackage{graphicx}
\usepackage{tikz}
\usepackage{mathrsfs}
\usepackage{amsmath}
\usepackage{amsfonts}
\usepackage{hyperref}

\usepackage{tikz}

\usepackage{marginnote,cite,amsmath,bbm}

 \usepackage{dsfont}
 \usepackage{psfrag}
 \usepackage{color}
 \usepackage{tikz}
 \usepackage{subfigure}
\usepackage{mathtools,amsthm,amssymb,amsfonts}
\usepackage{algorithm}
\usepackage{algorithmic}
\usepackage{tikz}
\usepackage{pgfplots}
\usepackage{siunitx}
\usepackage{psfrag}
\usepackage{color}
\usepackage{amssymb}
\usepackage{csquotes}
\usepackage{booktabs}
\usepackage[rightcaption]{sidecap}
\usepackage[colorinlistoftodos]{todonotes}
\usepackage{float}
\usepackage{libertine}
\usepackage{xurl}
\usepackage{csquotes}
\usepackage{graphics}
\usepackage{graphicx}
\def\R{\mathbb{R}}

\def\N{\mathbb{N}}

\graphicspath{{./convergence exp}{./convergence resonnance}{./optimization}{./comparaison methode}{./comparaison loop projection}{./comparaison gen}}

\def\del  {\partial}

\def\R{\mathbb{R}}

\def\N{\mathbb{N}}

 \def\dx{{\rm d}x}

\def\Tr{\operatorname{Tr}}

\newtheorem{remark}{\textbf{Remark}}

\makeatother
\theoremstyle{plain}
\newtheorem{proposition}{\textbf{Proposition}}
\newtheorem{theorem}{\textbf{Theorem}}
\newtheorem{lemma}{\textbf{Lemma}}
\newtheorem{definition}{\textbf{Definition}}

\usepackage{caption} 
\usepackage{fancyhdr}

\newcommand{\bR}{\mathbb{R}}
\newcommand{\bC}{\mathbb{C}}
\newcommand{\Arg}{\operatorname{Arg}}

\DeclareMathOperator{\supp}{supp}

\author{
  {\normalsize Nabil Alami, Edward Lucyszyn, Raphaël Pain Dit Hermier, Anna Rozanova-Pierrat}\thanks{CentraleSup\'elec, Universit\'e Paris-Saclay, France.}
    		}
\title{Parametric shape optimization for the convected Helmholtz equation with a generalized Myers boundary condition}
\date{}

\begin{document}
\maketitle
\thispagestyle{fancy}

\begin{abstract}
\noindent We consider the convected Helmholtz equation with a generalized Myers boundary condition (a boundary condition of the second order) and characterize the set of physical parameters for which the problem is weakly well-posed. The model comes from industrial applications to absorb acoustic noise in jet engines filled with absorbing liners (porous material). The problem is set on a 3D cylinder filled with a $d$-upper regular boundary measure, with a real $1<d\le 2$.  This setup leads to a parametric shape optimization problem, for which we prove the existence of at least one optimal distribution for any fixed volume fraction of the absorbing liner on the boundary that minimizes the total acoustic energy on any bounded wavenumber range.
\end{abstract}


\section{Introduction}	

We consider the convected Helmholtz equation~\eqref{EqHCphys} with a generalized version of Ingard-Myers boundary condition (a boundary condition of the second order, see~\eqref{EqBcGamma})~\cite{AUREGAN-2001} and start by characterizing the set of physical parameters for which the problem is well-posed on a 3D cylinder filled with a $d$-upper regular boundary measure, with a real $1<d\le 2$ (see~\eqref{defmu}). 
The model comes from industrial applications to absorb acoustic noise in jet engines filled with absorbing liners (porous material). 
A liner is a panel structure comprised of two layers: a top layer made of a porous material designed to absorb waves and a bottom layer made of an impervious (reflexive) material.

In the optimal energy absorption framework in aircraft engines, we address the following engineering problem: the existence of an optimal distribution of the liner of a small (to compare to the total engine's shape) fixed volume in the reflective material, minimizing the acoustical energy of the engine. The practical reason is to absorb the noise in the best way with a small quantity of a liner, which could be cheaper to compare to the reactor (of an initially fixed shape) entirely covered with liners. 

In this article, we show, by the techniques of parametric shape optimization, the existence of at least one optimal liner distribution for any fixed total volume, realizing the infimum of the acoustical energy (the minimum for the relaxation problem) inside of a cylindrical engine. This is the main result of the article. It is given in Theorem~\ref{ThOptShape} not only for a fixed noise wavenumber but also for a fixed wavenumber range. The studied shape of the reactor (see Fig.~\ref{cut}) is motivated by the physical experiment setup~\cite{RENOU-2011} in which the generalized Myers condition was initially introduced. On the boundary of the cylindrical engine, we fix a $d$-upper regular boundary measure, with a real $1<d\le 2$, previously used for the well-posedness of the model (see Theorem~\ref{th:well-posedness}). A typical example of such measure is the sum of the cylindrical Hausdorff surface measure and the Cantor set-type measure (see also Fig.~\ref{fig:cylinderconvergence} for a convergent sequence of domains having in the limit the cylindrical domain with this kind of boundary measure).
To our knowledge, the results on the well-posedness of this generalized Myers boundary condition (and in addition, in the presence of a $d$-upper regular boundary measure) and the optimal shape distribution of the liners have never been addressed before. However, the theoretical and numerical parametric shape optimization with different application goals is generally a very common subject as presented in~\cite{ALLAIRE-2007,ALLAIRE-1997,BANICHUK-1990,HASLINGER-2003,HENROT-2005,PIRONNEAU-1984, MAGOULES-2025}  and their references. For the geometrical shape optimization ($i.e.$ for the optimization of the boundary shape itself) for models with Robin-type conditions, we mainly refer to~\cite{BUCUR-2005,BUCUR-2016,HINZ-2021-1,HINZ-2023,MAGOULES-2021}


One of the well-known boundary conditions to model the sound interaction with the liner in the presence of uniform flow is the  Ingard–Myers boundary condition~\cite{INGARD-1959,MYERS-1980,RENOU-2011}, modeling the interaction of the acoustic wave with the lined wall. The Ingard-Myers boundary condition has been studied extensively primarily due to its significant industrial applications, particularly in minimizing acoustic noise in jet engines~\cite{REDON-2011,MATHEWS-2018}.   For instance, aircraft engines employ acoustic liners on the inner walls of the engine nacelle to reduce engine noise. These liners utilize the Helmholtz resonance principle to dissipate incoming acoustic energy~\cite{MA2020}.
However, several papers which describe its failures to accurately predict the liner's behavior have been published since the early 2000s. For example, theoretical evidence by Brambley~\cite{BRAMBLEY-2013} and experimental evidence by Renou and Aurégan~\cite{RENOU-2011}, showing discrepancies between downstream and upstream wave numbers, as well as significant differences between measured and predicted scattering matrices using the Myers-Ingard condition, further demonstrate its inadequacy for ducts with uniform flow assumptions. Particularly, viscous and turbulent effects near the wall can affect this boundary condition, especially at very low frequencies~\cite{AUREGAN-2001}. Other works have shown its instability and problem with convergence in the time domain~\cite{BRAMBLEY-2012,BRAMBLEY-2009}. A modified equation, which we called here by the generalized Myers condition, is introduced by Y. Renou and Y. Aurégan in~\cite{RENOU-2011} with an additional parameter $\beta_v$, which models the transfer of momentum into the lined wall induced by molecular and turbulent viscosities (see~\eqref{EqBcGamma}). In contrast to \cite{MAGOULES-2025}, the fluid motion follows the tangential direction to the boundary, which is not a favored direction for the absorption situation compared with the normal incidence case, following the famous Bardos-Lebeau-Rauch geometrical rays approach~\cite{BARDOS-1992,BARDOS-1994}. The model with $\beta_v=0$ corresponds exactly to the Ingard-Myers condition and was previously considered in 2D case~\cite{LUNEVILLE-2014}. We partially use it here for the well-posedness of the generalized model.
The well-posedness of plane models involving the convected Helmholtz equation with different boundary conditions of the first order is also well-known~\cite{BECACHE-2004,BONNET-2002,LUNEVILLE-2002}.


%
%
%

 The work with a class of $d$-upper regular boundary requires proper frameworks such as definitions of the trace operator and Green's formula~\cite{MAGOULES-2025,ROZANOVA-PIERRAT-2021,HINZ-2021,MAGOULES-2021}, presented in Section~\ref{sectionmeasure} in order to establish the variational formulation, obtained in details in~Appendix~\ref{sec:variational formula}.

We note that the second-order boundary condition using the external parameter $\beta_v$ introduces more difficulty. More experiments are needed to provide benchmark data on this $\beta_v$ factor~\cite{RENOU-2011}, not yet well known experimentally. By our well-posedness result in Theorem~\ref{th:well-posedness}, we provide benchmark values for the parameter $\beta_v$ for different behaviors of the liner's physical properties (impedance) in a specific case. In~Appendix~\ref{AppLimitGr}, we consider the limit behavior of admissible values of $\beta_v$ for well-posedness in the case where the imaginary part of the liner's admittance dominates its real part.

Once the weak well-posedness of the convected Helmholtz equation with a generalized version of Ingard-Myers boundary condition is established we address the parametric shape optimization problem and prove our main result of the existence of at least one optimal distribution of liners with a small total volume, realizing the infimum of the acoustical energy on any bounded segment of wavenumbers thanks to a relaxation method and the result on the energy continuity.

The outline of this paper is as follows. Section~\ref{sec2} introduces the physical model described by the convected Helmholtz equation with the generalized Myers boundary condition. In Section~\ref{sectionmeasure}, we introduce the functional framework allowing us to consider $d$-upper regular boundary measures and the main hypothesis for the well-posedness of the model. 
Section~\ref{sec:wellposedness} is dedicated to prove the existence and unicity of the weak solution (the details on the variational formulation are given in~Appendix~\ref{sec:variational formula}), while also providing a characterization of the values of the parameter \(\beta_v\) that ensures well-posedness (this part is completed in~Appendix~\ref{sec:variational formula}). Finally, in Section~\ref{sec:shapeoptimization}, we deal with the parametric shape optimization approach to demonstrate the existence of an optimal liner distribution of a small fixed quantity, minimizing the acoustic energy on all wavenumber bounded intervals.

\section{Model with generalized Myers boundary conditions}\label{sec2}
To model the wave propagation in the reactor, we define the cylindrical domain $\Omega \subseteq \bR^3$ in our study to be the following: if $\mathcal{B}(h,R)$ denotes the open ball in $\bR^2$ of center $h\in \bR^2$ and radius $R > 0$, then
\begin{equation}\label{EqOmega}
    \Omega := (0,L)\times \mathcal{B}(0_{\bR^2},R),
\end{equation}
where $L,R\in (0,+\infty)$ are arbitrary, and represent respectively the length and radius of the domain.

We then define the different parts of the boundary of $\Omega$ (see Fig.~\ref{cut}) 
\begin{equation}
    \label{Gammasdef}
    \Gamma_{in}:= \{0\}\times \overline{\mathcal{B}(0_{\bR^2},R)}, \quad \Gamma_{out}:= \{L\}\times \overline{\mathcal{B}(0_{\bR^2},R)}, \quad \Gamma:= [0,L]\times \del \mathcal{B}(0_{\bR^2},R). 
\end{equation}


\begin{figure}[h]
    \centering
    
        \centering
        \includegraphics[width=0.7\textwidth]{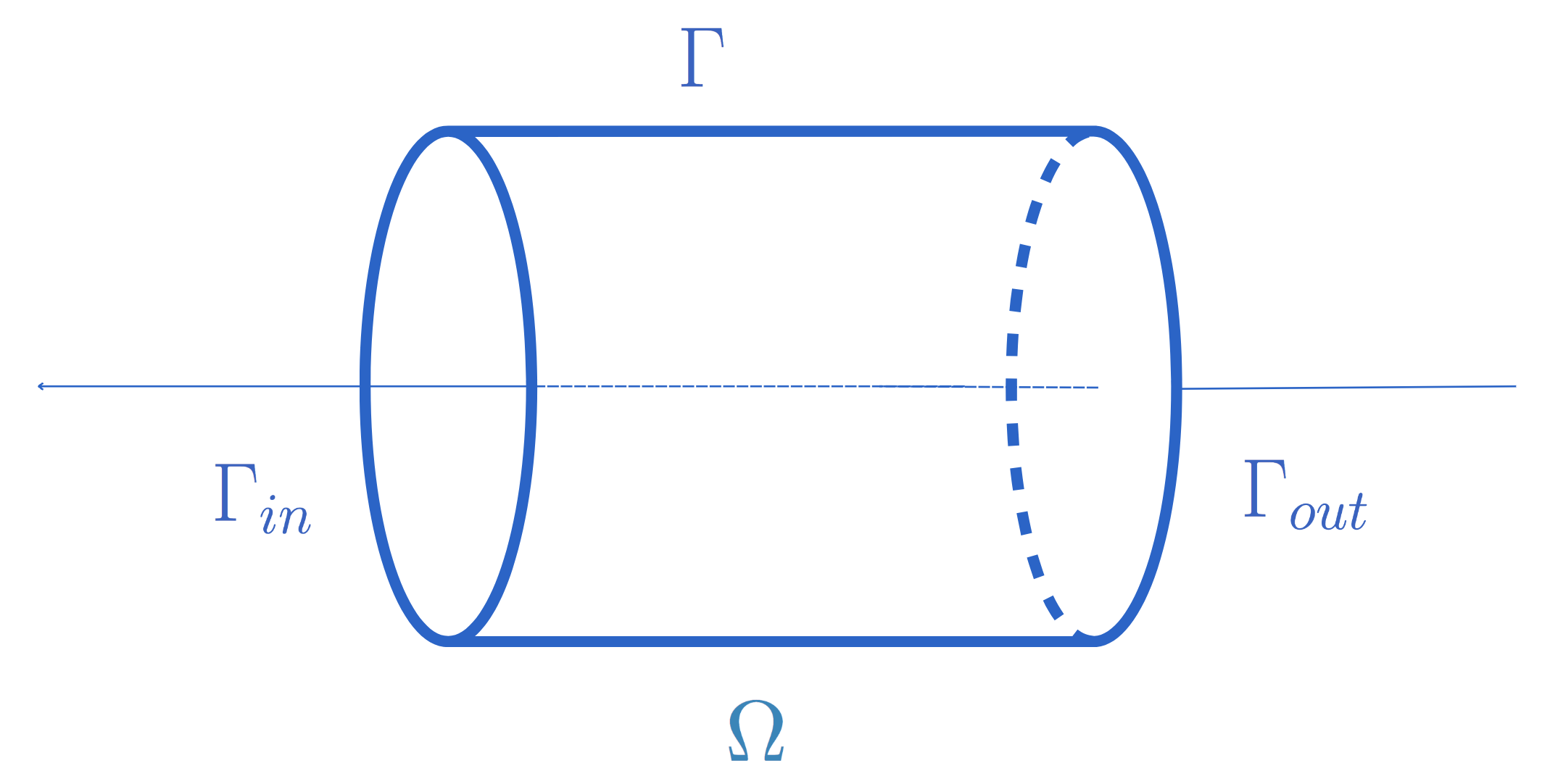}
    \caption{Cylinder modeling the reactor:  The boundary parts $\Gamma_{in}$ and $\Gamma_{out}$ represent respectively the zones of air inlet and outlet in the reactor. The boundary part $\Gamma$ represents the wall where the liner is located.} \label{cut}
\end{figure}
 The boundary parts $\Gamma_{in}$ and $\Gamma_{out}$ represent respectively the zones of air inlet and outlet in the reactor. The boundary part $\Gamma$ represents the wall where the liner is located. In the presence of a uniform flow along the principal axis $x$ of the cylinder, the perturbed pressure $p$ of an acoustic wave around a constant state in the harmonic regime is supposed to satisfy the following convected Helmholtz equation:
 \begin{equation}\label{EqHCphys}
 	 \Delta  p  + k_0^2\left(1-\frac{iM_0}{k_0}\del_x\right)^2 p = 0, \quad (x,y,z)\in \Omega. 
 	 \end{equation}
 Here $\Delta=\del^2_x+\del^2_y+\del^2_z$ is the usual Laplacian, 
$M_0$ is the Mach number, $k_0$ is the wave number of the acoustical wave. 
Eq.~\eqref{EqHCphys} is  the harmonic regime linear approximation of the Euler system for an adiabatic and incompressible fluid flow in the presence of a constant uniform flow along $x$-axis.
In what follows we use the notation
\begin{equation}\label{eqD}
    \displaystyle \mathcal{D} = k_0\left(1 - \frac{iM_0}{k_0}\partial_x\right),
\end{equation} 
to rewrite the convected Helmholtz equation in the form
 \begin{equation}\label{EqHC}
 	 \Delta  p  +  \mathcal{D}^2 p = 0, \quad (x,y,z)\in \Omega.
 	 \end{equation}
The inflow condition on $\Gamma_{in}$ reads by
\begin{equation}\label{EqGin}
	p|_{\Gamma_{in}}=g
\end{equation}
for some source $g$ modeling the incoming reactor noise, and the outflow condition on $\Gamma_{out}$ is given by the usual absorbing impedance condition 
\begin{equation}\label{EqGout}
	\left.\frac{\del p}{\del n}+ik p\right|_{\Gamma_{out}}=0,
\end{equation}
where $k=\displaystyle \frac{\omega}{u_0}$ is the wave number for the fluid with $u_0>0$ the constant velocity of the fluid along $x$-axis.
The boundary condition modeling the interaction with the liner on $\Gamma$ is given~\cite[Eq.~(24)]{RENOU-2011} by the following generalized Myers condition:
\begin{equation}\label{EqGMBC}
    \frac{\partial p}{\partial n} =\frac{\Upsilon}{i\omega }\bigg(i\omega +(1-\beta_v)u_0\frac{\partial}{\partial x}\bigg)\bigg(i\omega +u_0\frac{\partial}{\partial x}\bigg)p,
    \end{equation}
where $\displaystyle \Upsilon = \frac{Z_0}{c_0Z}$, $c_0$ is the sound speed in the fluid, $Z_0  > 0$ is the fluid impedance and $Z$ is the impedance of the liner, supposed here to be a known complex-valued function of the frequency $\omega$ with a strictly positive real part.
 In Eq.~\eqref{EqGMBC}, $\beta_v$ is a complex number with modulus strictly less than $1$ that models the transfer of momentum to the wall with the liner caused by kinematic and turbulent viscosities. With $\beta_v = 0$, we recover the Ingard-Myers condition. The presence of this complex coefficient experimentally improves the mathematical model~\cite{RENOU-2011}. We notice that the boundary condition~\eqref{EqGMBC}, as the convected Helmholtz equation, mainly depends on the moving properties of the fluid along the $x$-axis, which is one of the tangential directions to the boundary. 
However, the relevant values of $\beta_v$ are not yet estimated experimentally~\cite{RENOU-2011} and we give them in Theorem~\ref{th:well-posedness}.


With our notation~\eqref{eqD}, the boundary condition~\eqref{EqGMBC} on $\Gamma$ becomes:
\begin{equation}\label{EqBcGamma}
    \displaystyle \frac{\partial p}{\partial n} + iY\frac{Z_0}{k_0} \mathcal{D}(\mathcal{D} + iM_0\beta_v\partial_x) p  = 0\text{ on }\Gamma
\end{equation}
which is, as mentioned previously, the usual Ingard-Myers condition when $\beta_v=0$. Here we denote by $\displaystyle Y = \frac{1}{Z}$ the complex-valued admittance of the liner with the real part $\Re e(Y)>0$ strictly positive. As $Z$, the admittance $Y$ also can take different values depending only on $w$.

We notice the following decomposition of the second order operator in~\eqref{EqBcGamma}:
\begin{equation} \label{squarediff}\mathcal{D}(\mathcal{D} + iM_0\beta_v\partial_x) = \mathcal{D}_1^2-K^2\end{equation}
 by defining the notations
\begin{equation} \label{D1andK} \mathcal{D}_1:=\alpha \Big(\frac{k_0}{2}\frac{2-\beta_v}{1-\beta_v}-iM_0\partial_x\Big)\hspace{2pt}; \quad
\hspace{2pt}K:= k_0\frac{\beta_v}{2\alpha}\hspace{2pt}\end{equation}
with $\alpha^2 = 1-\beta_v$ and $\arg(\alpha)\in [0,\pi)$. In fact, this decomposition will be useful
for finding an adequate Fredholm-type decomposition (described in subsection~\ref{sub:fredholm}) to our variational formulation.

Let $\mathbf{x}\in \Gamma$. To model the partial presence of liners on $\Gamma$, we define the distribution of the liner on $\Gamma$ by the characteristic function $\chi: \Gamma \to \{0,1\}$, with $\chi(\mathbf{x}) = 1$ if the liner is at $\mathbf{x}$, and $\chi(\mathbf{x}) = 0$ otherwise. Therefore, instead of~\eqref{EqBcGamma}  modeling the presence of the liner on all shape of $\Gamma$, we consider
\begin{equation}\label{EqBcGammaChi}
    \displaystyle \frac{\partial p}{\partial n} + i\chi Y\frac{Z_0}{k_0} \mathcal{D}(\mathcal{D} + iM_0\beta_v\partial_x) p  = 0\text{ on }\Gamma.
\end{equation}
If for $\mathbf{x}\in \Gamma$ $\chi(\mathbf{x})=0$, then condition~\eqref{EqBcGammaChi} for this boundary point becomes the homogeneous Neumann condition $\displaystyle \frac{\partial p}{\partial n}=0$ and hence, this means the reflection in the liner absence. 

In the next section we define the weak framework for the introduced model and precise in which sense we understand the boundary conditions~\eqref{EqGin},~\eqref{EqGout} and~\eqref{EqBcGammaChi}. In what follows, we fix all introduced previously physical constants 
\begin{equation}\label{EqPhysConstants}
	\omega > 0,\; u_0 > 0,\; c_0 > 0,\; Z_0 > 0, \;Z \in \{z \in \mathbb{C} \hspace{2pt} | \hspace{2pt} \Re e(z) > 0\},\;M_0,\;\displaystyle k_0\; \hbox{and } k.
\end{equation}

\section{Functional framework} 
\label{sectionmeasure}
\noindent As $\Omega$ is the cylindrical domain defined previously, its  boundary $\partial\Omega$ is Lipschitz, which is an example of a $2$-(Alfors regular) set~\cite{LUNEVILLE-2014}. The typical measure on Lipschitz boundaries of a domain of $\R^3$ would be $2$-dimensional Lebesgue or Hausdorff measure. Instead of it,  we consider more general boundary measures. For a real $d\in (1,3)$ we fix a $d$-upper regular positive Borel measure $\mu$ on the boundary $\del \Omega$, that is a measure $\mu$ on $\bR^3$ which satisfies:
\begin{equation}
    \label{defmu}
    \begin{cases}
        \text{supp }\mu = \partial \Omega,\\
        \exists A > 0, \forall \mathbf{x}\in \partial \Omega, \forall r \in (0,1], \; \mu(\mathcal{B}(\mathbf{x},r)) \leq Ar^d.
    \end{cases}
\end{equation}
Condition~\eqref{defmu} implies that the Hausdorff dimension of $\del \Omega$ must be bigger or equal to $d$. Therefore, for our cylindrical case, we consider only $d\in(1,2]$. 
Let us recall that a particular example of a $d$-upper regular Borel measure is a $d$-measure (or $d$-dimensional measure) satisfying 
in addition the lower regularity property with the same $d$: there exist $A$ and $B>0$ such that
\begin{equation*}
     \forall \mathbf{x}\in \partial \Omega, \forall r \in (0,1], \; Br^d\le \mu(\mathcal{B}(\mathbf{x},r)) \leq Ar^d.
\end{equation*}
Example of a $2$-upper regular measure $\mu$ with $\operatorname{supp} \mu =\del \Omega$ is the sum of the $2$-dimensional Hausdorff measure of $\del \Omega$ and  a $d$-dimensional measure with $\displaystyle d \in (1,2)$ with a support included in $\del \Omega$. 
The resulting sum is thus a $2$-upper regular measure thanks to the following proposition.
\begin{proposition}
    \label{upperprop}
   Let $F$ be a Borel set of $\R^n$, $\mu$ be a measure on $\bR^n$ with $\text{supp }\mu = F$, and $d_1 < d_2 \in (n-2,n)$. If $\mu$ is a $d_2$-upper regular measure for $F$, then it is a $d_1$-upper regular measure.
\end{proposition}
\begin{proof}
    Suppose $\mu$ is a $d_2$-upper regular measure for $F$, and let $A > 0$ be the constant from \eqref{defmu} given by $d_2$-upper regularity. Then, for any $\mathbf{x}\in F$ and $r \in (0,1]$, we have
$\mu(\mathcal{B}(\mathbf{x},r)) \leq Ar^{d_2} \leq Ar^{d_1}$, since $\displaystyle \frac{r^{d_2}}{r^{d_1}} = r^{d_2-d_1} \in (0,1]$ because $d_2-d_1 > 0$.
\end{proof}
As shown below in Fig.~\ref{fig:cylinderconvergence}, one could consider for example a sequence of $C^\infty$ domains with boundaries equipped with the usual $2$-dimensional Hausdorff measure, converging to our domain $\Omega$ whose boundary is equipped with a measure different from the $2$-dimensional Hausdorff measure. This measure would here be the sum of the $2$-dimensional Lebesgue measure and a $d$-dimensional measure, with $\displaystyle d = 1+\frac{\log(2)}{\log(3)} \in (1,2)$, with support being the revolution of a scaled Cantor set along the $x$ axis. The resulting sum is thus a $d$-upper regular measure thanks to Proposition~\ref{upperprop}.

\begin{figure}[!ht]
    \centering
    \includegraphics[width=1\linewidth]{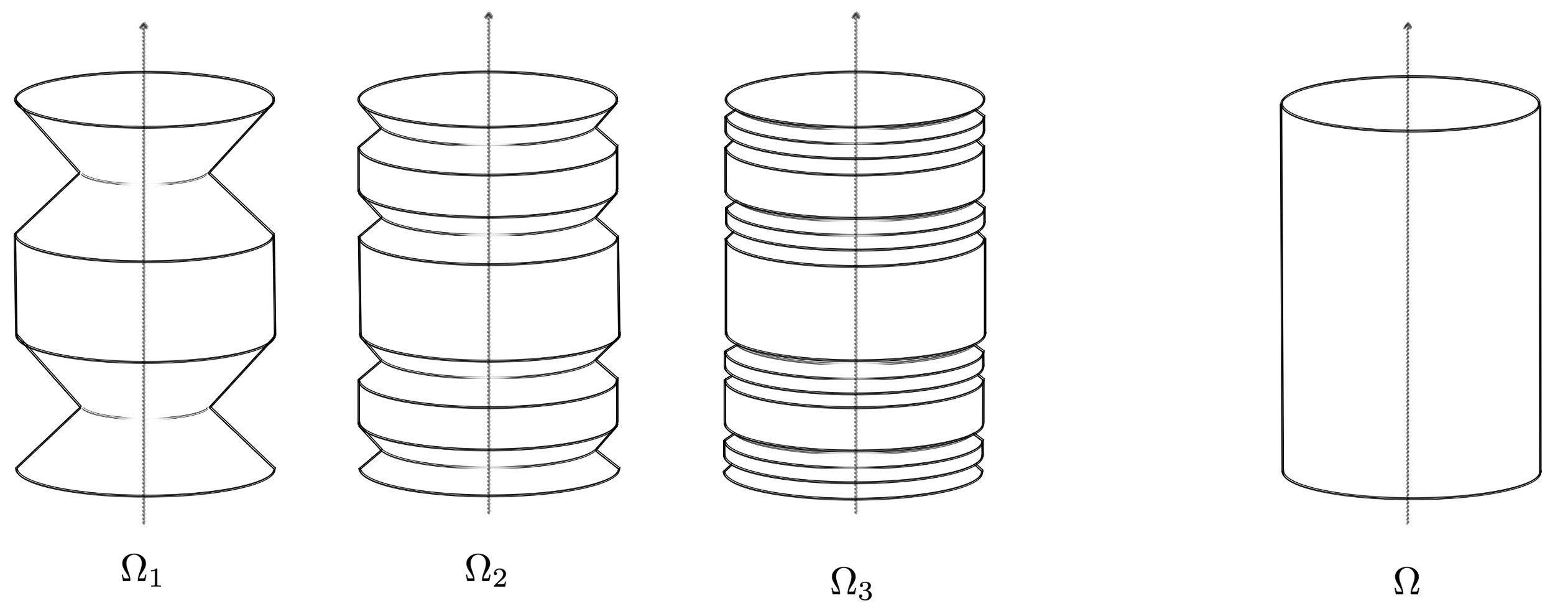}
    \caption{An example sequence of Lipschitz 2D boundaries $(\Omega_n)_{n\in\mathbb{N}}$ converging to the cylinder $\Omega$ not equipped with $2$-dimensional Hausdorff measure.}
    \label{fig:cylinderconvergence}
\end{figure}
Equipped with a fixed $d$-upper regular Borel measure $\mu$ on the boundary for $d\in (1,2]$ with $\operatorname{supp} \mu =\del \Omega$, we define
the space $L^2(\del \Omega,\mu)$ as the space of measurable functions on $\del \Omega$ such that $\|h\|_{L^2(\del \Omega,\mu)}=\sqrt{\int_{\del \Omega} |h|^2d \mu}$ is finite.
\begin{remark}
This pair $(\Omega,\mu)$ is thus a particular example of a Sobolev admissible domain, defined~\cite{HINZ-2021,MAGOULES-2025} as a
bounded domain (open and connected set) $\Omega\subset \R^n$, (here $n=3$), which are 
\begin{enumerate}
	 \item[(a)]  $H^1$-extension domains (the cylindrical domain is a Lipschitz domain and thus it is $H^1$-extension domain);
	\item[(b)] its boundary $\del \Omega$ is the support of finite positive Borel $d$-upper regular measure for a fixed real number $d\in (n-2,n)$.
\end{enumerate}
\end{remark}
Then we suppose 
that 
\begin{equation}\label{EqDivisionDelOmega}
 	\mu(\Gamma_{out}\cap \Gamma)=\mu (\Gamma_{in}\cap \Gamma)=\mu(\Gamma_{in}\cap \Gamma_{out})=0,
 \end{equation}
and  $\Gamma_{out}$, $\Gamma_{in}$ and $\Gamma$ are closed subsets of $\del \Omega$.
As $\Gamma$ is composed of a liner and reflexive parts (see~\eqref{EqBcGammaChi}), to avoid degenerate cases,  we suppose that each part of $\Gamma$, with a liner or without, has positive capacity with respect to the space $H^1(\mathbb{R}^3)$ (see for instance~\cite[Section 7.2]{MAZ'JA-1985}) and has a strictly positive value of the measure $\mu$. Up to a zero $\mu$-measure set, the part of $\Gamma$ filled with the liner/porous material can be considered as its compact subset. 

The assumptions that $\Gamma_{out}$, $\Gamma_{in}$ and $\Gamma$ and its liner part are closed in the induced topology on $\del \Omega$ ensure that the linear trace operators $\operatorname{Tr}_{\Gamma_{out}}: H^1(\Omega)\to L^2(\Gamma_{out},\mu)$, $\operatorname{Tr}_{\Gamma_{in}}: H^1(\Omega)\to L^2(\Gamma_{in},\mu)$ and $\operatorname{Tr}_{\Gamma}: H^1(\Omega)\to L^2(\Gamma,\mu)$ 
 are compact (for their definitions see~\cite{HINZ-2021,HINZ-2021-1,HINZ-2023} initially adopted from~\cite[Corollaries 7.3 and 7.4]{BIEGERT-2009} and based on the restriction of  quasi-continuous representatives of $H^1(\R^3)$-elements.

 The basic properties of the trace operator are presented in~\cite[Corollary~5.2]{HINZ-2021} and detailed in~\cite{CLARET-2023}. As we work in the particular case of a bounded Sobolev admissible domain, then we also have the following compactness result:
 \begin{theorem}
 	Let $\Omega\subset \R^3$ be the bounded cylindrical domain defined in~\eqref{EqOmega} and $\mu$ be a fixed $d$-upper regular positive Borel measure with $\operatorname{supp}\mu=\del \Omega$ and $d\in (1,2]$.
Then the image of the trace operator $B(\del \Omega):=\operatorname{Tr}_{\partial \Omega}(H^1(\Omega))$ endowed with the norm
\begin{equation}\label{normTri}
\left\|h\right\|_{B(\del \Omega)}:=\min\{ \left\|v\right\|_{H^1(\Omega)}|\ h=\mathrm{Tr}_{\partial\Omega}\:v\},
\end{equation}
is a Hilbert space, dense and compact in $L^2(\del \Omega,\mu)$.
 \end{theorem}
Let us notice that the definition of the image of the trace does not depend on the choice of the boundary measure. In particular, for the cylindrical domain $\Omega$ the norm $\|\cdot\|_{B(\del \Omega)}$ is equivalent to the norm $\|\cdot\|_{H^\frac{1}{2}(\del \Omega)}$. The measure dependence comes in the $L^2$-boundary framework.
In particular, it is also important in the usual Gelfand triple
 $B (\del \Omega)\subset L^2(\del \Omega,\mu) \subset B'(\del \Omega)$, where by  $B'(\del \Omega)$ is denoted the topological dual space of $B(\del \Omega)$.
 This construction allows us to define the normal derivative in the general sense:
 \begin{definition}
 	For all \( u \in H^1(\Omega) \) with \( \Delta u \in L^2(\Omega) \), the bounded linear functional \( \displaystyle \frac{\partial u}{\partial n} \in B'(\partial \Omega) \) is called the normal derivative of $u$ on $\del \Omega$ if it is defined for all $v\in H^1(\Omega)$ by
\begin{equation}
\label{greenformula}
\left\langle \frac{\partial u}{\partial n}, \operatorname{Tr}_{\partial \Omega} v \right\rangle_{(B'(\partial \Omega), B(\partial \Omega))} := \int_{\Omega} (\Delta u) v \, dx + \int_{\Omega} \nabla u \cdot \nabla v \, dx.
\end{equation}
This is the
generalized Green formula.
Similarly, for all $u\in H^1(\Omega)$ we define a bounded linear functional $u\cdot n_x\in B'(\del \Omega)$ for all $v\in H^1(\Omega)$ by
\begin{equation}\label{EqGFnx}
	\left\langle u\cdot n_x, \operatorname{Tr}_{\partial \Omega} v \right\rangle_{(B'(\partial \Omega), B(\partial \Omega))} := \int_{\Omega} \del_x u v \, dx + \int_{\Omega} u \del_x v \, dx.
\end{equation}
 \end{definition}
 \begin{remark}
\label{remarkregularity}
 If the normal derivative is more regular as just $B'(\del \Omega)$, but belongs to $L^2(\del \Omega,\mu)$, for instance, by the impedance boundary condition on $\Gamma_{out}$,
 $\displaystyle \frac{\partial p}{\partial n} =- ik \Tr_{\Gamma_{out}}p \in L^2(\Gamma_{out},\mu)$, then we have $$\left\langle\frac{\partial u}{\partial n}, \operatorname{Tr}_{\partial \Omega} v \right\rangle_{(B'(\partial \Omega), B(\partial \Omega))}=\int_{\del \Omega}\frac{\partial u}{\partial n}\operatorname{Tr}_{\partial \Omega} \bar v d \mu.$$

Similarly, if the functional $u\cdot n_x$ is more regular as just $B'(\del \Omega)$, but belongs to $L^2(\del \Omega,\mu)$, and if $\partial \Omega$ is Lipschitz, then $n_x \in L^\infty(\partial \Omega, \mu)$ can also be interpreted as the normal vector component along the $x$ axis, and we have
$$\left\langle u\cdot n_x, \operatorname{Tr}_{\partial \Omega} v \right\rangle_{(B'(\partial \Omega), B(\partial \Omega))} = \int_{\partial \Omega} \Tr_{\partial \Omega}u \Tr_{\partial \Omega} \bar v \cdot n_x \, d\mu.$$
\end{remark}

\section{Well-posedness of the model}
\label{sec:wellposedness}
In this section, we  prove  the weak well-posedness of the introduced model, using the Fredholm alternative and updating the usual methodology~\cite{MAGOULES-2025,BONNET-2002,LUNEVILLE-2014}. 
Instead of non-homogeneous Dirichlet boundary conditions, we consider, after the standard removal method, the following problem with the non-homogeneous source terms $f\in L^2(\Omega)$ and $\eta \in L^2(\Gamma,\mu)$:
\begin{equation}
	 \begin{cases}
    \displaystyle \Delta  p  + \mathcal{D}^2 p = f \in L^2(\Omega),
    \\
    \displaystyle \frac{\partial p}{\partial n} + iY\frac{Z_0}{k_0}\chi \Tr_\Gamma \Bigl[ \mathcal{D}(\mathcal{D} + iM_0\beta_v\partial_x) p \Bigr] = \eta ,
    \\
    \displaystyle \Tr_{\Gamma_{in}}p = 0,
    \\
    \displaystyle \frac{\partial p}{\partial n} + ik \Tr_{\Gamma_{out}}p = 0.
    \end{cases} \label{PH}
\end{equation}
Here, $\chi\in L^\infty(\Gamma,\mu)$ be a nonnegative and bounded Borel function on $\Gamma$ which is positive with a positive minimum on a subset positive $\mu$-measure. 
we define the sesquilinear form:
\begin{center}
\begin{equation}
\label{chiproduct}
\langle \cdot , \cdot \rangle_{\chi}: \begin{cases}(L^2(\Gamma,\mu))^2 \longrightarrow  \mathbb{C} \\
  (g, g')  \longmapsto  \displaystyle \int_\Gamma \chi g \overline{g'} d\mu\end{cases}
\end{equation}
\end{center}
as well as the space
\begin{equation}
    V(\Omega) = \{q \in H^1(\Omega) \mid \Tr_{\Gamma_{in}}q = 0, \, \mathcal{D}_1q \in H^1(\Omega) \} \label{Space},
\end{equation}
endowed with the norm
\begin{equation} \label{Vnorm}\|p\|^2_{V(\Omega)} = \|\nabla p\|^2_{(L^2(\Omega))^3}+ \|\Tr_\Gamma(\mathcal{D}_1p)\|^2_{\chi}.\end{equation}
Here $\mathcal{D}_1$ is the differential operator defined in~\eqref{D1andK}.
When deemed appropriate (for example when dealing with multiple distributions $\chi$), we write the previous norm as $\|\cdot\|_{V(\Omega),\chi}$ to point out the dependence in $\chi$.

Additionally, we  show that $V(\Omega)$ is a Hilbert space by proving that it is the space of weak solutions of the following boundary-value problem: 
\[
    (P_{h,\varphi,\psi}) \hspace{2pt} : \hspace{2pt}
    \begin{cases}
    \displaystyle \partial_x^2 q = h \in L^2(\Omega),
    \\
    \displaystyle \displaystyle  \partial_x q\cdot n_x - iY\frac{Z_0}{k_0}K^2\chi \Tr_{\Gamma}  q  = \varphi\in L^2(\Gamma,\mu), 
    \\
    \displaystyle \Tr_{\Gamma_{in}}q = 0,
    \\
    \displaystyle \frac{\partial q}{\partial n} + ik \Tr_{\Gamma_{out}}q = \psi\in L^2(\Gamma_{out},\mu).
    \end{cases} \label{Pgphipsi}
\]
The weak solutions of~\eqref{Pgphipsi} satisfy the variational formulation on $H^1_{\Gamma_{in}}(\Omega)=\{v\in H^1(\Omega)|\, \Tr_{\Gamma_{in}}v=0\}$, denoted by $(FV_{h,\varphi,\psi})$: 
\begin{multline*}
 \forall v\in H^1_{\Gamma_{in}}(\Omega),\; (\partial_xq,\partial_x v)_{L^2(\Omega)} +iY\frac{Z_0}{k_0}K^2\langle\Tr_\Gamma  q,\Tr_\Gamma v\rangle_{\chi}-ik(\Tr_{\Gamma_{out}} q, \Tr_{\Gamma_{out}} v)_{L^2(\Gamma_{out},\mu)}\\ = -(h,v)_{L^2(\Omega)} + (\varphi,\Tr_{\Gamma} v)_{L^2(\Gamma,\mu)} + (\psi,\Tr_{\Gamma_{out}} v)_{L^2(\Gamma_{out},\mu)},
 \end{multline*}
 which is well-posed for any triple $(h,\varphi,\psi)\in L^2(\Omega)\times L^2(\Gamma,\mu) \times L^2(\Gamma_{out},\mu)$. This ensures
\begin{align}
\label{VomegaFV}
\begin{split}
V(\Omega) = &\Big\{q\in H^1(\Omega) \mid \, \mathcal{D}_1q \in H^1(\Omega), \Tr_{\Gamma_{in}}q = 0,\\
&\exists (h,\varphi,\psi)\in L^2(\Omega)\times L^2(\Gamma,\mu) \times L^2(\Gamma_{out},\mu), \, q \text{ verifies } (FV_{h,\varphi,\psi})\Big\}    
\end{split}
\end{align}
is a closed subset of $H^1(\Omega)$, thus a Hilbert space.

\begin{proposition}[Variational Formulation]
\label{prop:FV}
The variational formulation associated with \eqref{PH} is given by
\begin{equation}
\label{eq:FV}
    \forall q\in V(\Omega), \ A(p,q) = \ell(q);
\end{equation}
where the forms $ A : V(\Omega)^2 \to \mathbb{C}$ and $ \ell : V(\Omega) \to \mathbb{C}   $ are defined by:
\begin{equation}\label{formL}\forall q\in V(\Omega), \, \ell(q) = (\eta,\Tr_\Gamma q)_{L^2(\Gamma,\mu)}-(f,q)_{L^2(\Omega)},\end{equation}
and $\forall p,q\in V(\Omega),$
\begin{align}
\label{formA}
\begin{split}
    A(p,q) = &\;(\nabla p,\nabla q)_{(L^2(\Omega))^3} - (\mathcal{D}p,\mathcal{D}q)_{L^2(\Omega)}+ik(\Tr_{\Gamma_{out}}p,\Tr_{\Gamma_{out}}q)_{L^2(\Gamma_{out},\mu)} \\
    &+ \ iY \frac{Z_0}{k_0}\Big[\langle \Tr_{\Gamma}(\mathcal{D}_1 p),\Tr_{\Gamma}(\mathcal{D}_1 q)\rangle_{\chi} - K^2\langle \Tr_{\Gamma}p,\Tr_{\Gamma}q \rangle_{\chi} \Big].
\end{split}
\end{align}
\end{proposition}
The proof of this proposition is given for the reader convenience in~Appendix~\ref{sec:variational formula}.
During the remainder of this section, we prove our first main result on the well-posedness of~\eqref{eq:FV}. 
\begin{theorem}[Weak well-posedness]
\label{th:well-posedness}
Let $\Omega\subset \R^3$ be the bounded cylindrical domain defined in~\eqref{EqOmega} and $\mu$ be a fixed boundary $d$-upper regular positive Borel measure for $d\in (1,2]$ and $\operatorname{supp}\mu=\del \Omega$. 
Assume $\del \Omega=\Gamma_{in}\cup\Gamma_{out}\cup \Gamma$ defined in~\eqref{Gammasdef} such that it holds~\eqref{EqBcGamma}. In particular, let $\Gamma$ be non trivial part of $\del \Omega : \mu(\Gamma)>0$, as well as the generalized Myers boundary condition on it: let $\chi\in L^\infty(\Gamma,\mu)$ be a nonnegative and bounded Borel function on $\Gamma$ which is positive with a positive minimum on a subset positive $\mu$-measure.
Assume in addition  $\beta_v\in \mathbb{C}$ such that $|\beta_v| < 1$, which either equals $\beta_v = 0$ or satisfies 
\begin{equation}
\label{betavcondition}
    \Re e(Y)\Re e(K^2) - \Im m(Y)\Im m(K^2) < 0,
\end{equation}
with $\displaystyle Y = \frac 1Z$ the liner admittance, and $K$ defined in \eqref{D1andK} and when squared gives
\[K^2=k_0^2\frac{\beta_v^2}{4(1-\beta_v)}.\]

Then for all $f\in L^2(\Omega)$, $\eta\in L^2(\Gamma,\mu)$,  and fixed values of all physical constants from~\eqref{EqPhysConstants} 
there exists 
a unique solution $p\in V(\Omega)$ of the variational formulation \eqref{eq:FV} of \eqref{PH}.

Moreover, the solution  $p\in  V(\Omega)$ continuously depends on the data:
there exists a constant $\hat{C}>0$, depending only on  $d$, $c_d$, $R$ and $L$ from~\eqref{EqOmega}, $\chi$, $\beta_v$ and other physical constants from~\eqref{EqPhysConstants},  
such that
 \begin{equation}\label{estimatebound}
  \|p\|_{V(\Omega),\chi}\le \hat{C}(\|f\|_{L^2(\Omega)}+ \|\eta\|_{L^2(\Gamma,\mu)}).
\end{equation}
\end{theorem}
\begin{remark}
We denote by admissible zone $\mathscr{B}_v$:
\begin{equation}
    \label{poissondef}
    \mathscr{B}_v =\{\beta_v\in \bC \mid |\beta_v| < 1, \Re e(Y)\Re e(K^2) - \Im m(Y)\Im m(K^2) < 0\}.
\end{equation}
Since this set only depends on the value of the ratio $\displaystyle r:= \frac{\Im m(Y)}{\Re e(Y)}\in \bR$, this set will be referred to as $\mathscr{B}_{v,r}$ when its dependence on $r$ needs to be mentioned (see Figs~\ref{fig:fish1},~\ref{fig:fish2} and~\ref{fig:two_images} with~Appendix~\ref{AppLimitGr}).
\end{remark} 
The next two subsections are dedicated to the proof of this theorem.
\subsection{Fredholm Decomposition}
\label{sub:fredholm}
We start by performing a Fredholm-type decomposition on~\eqref{eq:FV}.
By ``Fredholm-type decomposition" we understand here a decomposition of the following type:
\[A(p,q) = \Theta(p,q) + \langle K'u,v \rangle,\]
where $\Theta:(V(\Omega))^2\to \mathbb{C}$ is a continuous coercive sesquilinear form and $K':V(\Omega)\to V(\Omega)$ is compact. If $A$ admits such a decomposition, it can be transformed into \[\forall u,v\in V(\Omega), A(u,v) = \langle (cId-K')u,v \rangle\] with $c\ne 0$, up to isomorphism and change of inner product.

In this aim we write
\begin{equation}
\forall p,q\in V(\Omega), \quad A(p,q) = \Theta(p,q) + \xi(p,q),
\end{equation}
where
\begin{equation}
\label{Theta}
    \Theta(p,q) = (\nabla p,\nabla q)_{(L^2(\Omega))^3} - M_0^2(\partial_xp,\partial_xq)_{L^2(\Omega)}+iY \frac{Z_0}{k_0}\langle \Tr_\Gamma(\mathcal{D}_1 p),\Tr_\Gamma(\mathcal{D}_1 q)\rangle_{\chi},
\end{equation}
and
\begin{align}
\label{smallxi}	
\begin{split}
    \xi(p,q) = &-k_0^2(p,q)_{L^2(\Omega)}+iM_0k_0((\partial_xp,q)_{L^2(\Omega)}-(p,\partial_xq)_{L^2(\Omega)})\\
    &+ik(\Tr_{\Gamma_{out}}p,\Tr_{\Gamma_{out}}q)_{L^2(\Gamma_{out},\mu)}-iY \frac{Z_0}{k_0}K^2\langle \Tr_\Gamma p,\Tr_\Gamma q \rangle_{\chi}.
\end{split}
\end{align}
 Forms $\Theta$ and $\xi$ are clearly sesquilinear and continuous on $V(\Omega)$. Therefore, we apply the Riesz representation theorem to obtain:
\begin{equation}
\label{Xi}
\forall q,h\in V(\Omega),\quad \xi(q,h) = (\Xi q,h)_{V(\Omega)}, \end{equation}
where $\Xi:V(\Omega)\to V(\Omega)$ is a continuous linear operator.
\begin{lemma}
\label{le:xi compact}
The linear operator $\Xi:V(\Omega)\to V(\Omega)$ defined by (\ref{Xi}) is compact.
\end{lemma}
\begin{proof} We consider a weakly convergent sequence $q_n\rightharpoonup q$ in $V(\Omega)$. We directly have $\Xi q_n\rightharpoonup \Xi q$ in $V(\Omega)$ by continuity.
Then, by the compactness of the trace operator, $\Tr_{\partial \Omega} q_n\longrightarrow \Tr_{\partial \Omega} q$ in $L^2(\partial\Omega,\mu)$, hence in particular in $L^2(\Gamma_{out},\mu)$ and $L^2(\Gamma,\mu)$ ($\Gamma$ and $\Gamma_{out}$ are compact parts of $\partial \Omega$). Finally, the canonical injection $\iota:V(\Omega)\to L^2(\Omega)$ is compact because $V(\Omega)$ is a closed subspace of $H^1(\Omega)$.

We deduce by composition of continuous operators with compact/continuous operators:
\[\Xi q_n\longrightarrow \Xi q \text{ in } L^2(\Omega)\hspace{2pt};\hspace{2pt}\Tr_{\partial \Omega}(\Xi q_n)\longrightarrow \Tr_{\partial \Omega}(\Xi q) \text{ in } L^2(\partial\Omega,\mu)\hspace{2pt};\hspace{2pt}\partial_x\Xi q_n\rightharpoonup \partial_x\Xi q\text{ in } L^2(\Omega).\]
Then,
\begin{align*}
\begin{split}
  \|\Xi q_n\|^2_{V(\Omega)} &= -k_0^2(q_n,\Xi q_n)_{L^2(\Omega)}+iM_0k_0((\partial_xq_n,\Xi q_n)_{L^2(\Omega)}-(q_n,\partial_x\Xi q_n)_{L^2(\Omega)}) \\
  &+ \ ik(\Tr_{\Gamma_{out}}q_n,\Tr_{\Gamma_{out}}(\Xi q_n))_{L^2(\Gamma_{out},\mu)} -iY \frac{Z_0}{k_0}K^2\langle \Tr_\Gamma q_n,\Tr_\Gamma(\Xi q_n) \rangle_{\chi};
\end{split}
\end{align*}
and thus:
\begin{align*}
\begin{split}
 \|\Xi q_n\|^2_{V(\Omega)} \longrightarrow &-k_0^2(q,\Xi q)_{L^2(\Omega)}+iM_0k_0((\partial_xq,\Xi q)_{L^2(\Omega)}-(q,\partial_x\Xi q)_{L^2(\Omega)}) \\
  &+ \ ik(\Tr_{\Gamma_{out}}q),\Tr_{\Gamma_{out}}(\Xi q))_{L^2(\Gamma_{out},\mu)} -iY \frac{Z_0}{k_0}K^2\langle \Tr_\Gamma q,\Tr_\Gamma(\Xi q) \rangle_{\chi}.
\end{split}
\end{align*}
Therefore, $\|\Xi q_n\|^2_{V(\Omega)}\longrightarrow \|\Xi q\|^2_{V(\Omega)}$. The weak convergence in $V(\Omega)$ coupled with the convergence in norm allows us to conclude that $q_n \longrightarrow q$ in $V(\Omega)$, thus proving that $\Xi$ is compact.\end{proof} 
\begin{lemma}
\label{le:theta coercive}
The sesquilinear form $\Theta : V(\Omega)^2 \to \mathbb{C}$ defined by (\ref{Theta}) is coercive.
\end{lemma}
\begin{proof}
Let $p\in V(\Omega)$. We have:
\begin{equation}\Theta(p,p) = \|\nabla p\|^2_{(L^2(\Omega))^3} - M_0^2\|\partial_xp\|^2_{L^2(\Omega)}+iY \frac{Z_0}{k_0} \|\Tr_\Gamma(\mathcal{D}_1p)\|^2_{\chi}.\end{equation}
We denote \(iY =  -|Y| e^{i\theta}\), where \(\displaystyle\theta \equiv \Arg(Y) - \frac{\pi}{2} [2\pi]\). We then define:
\begin{equation}\lambda := \|\nabla p\|^2_{(L^2(\Omega))^3} - M_0^2\|\partial_xp\|^2_{L^2(\Omega)}\hspace{2pt};\hspace{2pt} \beta :=|Y| \frac{Z_0}{k_0} \|\Tr_\Gamma(\mathcal{D}_1p)\|^2_{\chi}.\end{equation}
By writing $\Theta$ in the following form, inspired by~\cite{LUNEVILLE-2014}:
\begin{equation}
|\Theta(p, p)|^2 = | \lambda - e^{i\theta}\beta|^2 = (\lambda - \beta)^2 + 4\lambda\beta \sin^2\left(\frac{\theta}{2}\right) \geq \sin^2\left(\frac{\theta}{2}\right) (\lambda + \beta)^2,\end{equation}
one can deduce that,
\begin{equation}|\Theta(p, p)| \geq \Big|\sin\left(\frac{\theta}{2}\right)\Big|\min\Big(1-M_0^2,|Y| \frac{Z_0}{k_0}\Big)\Big|\|p\|^2_{V(\Omega)}.\end{equation}
Moreover, $\displaystyle\sin\left(\frac{\theta}{2}\right)\neq 0$ (otherwise $\displaystyle \Arg(Y) \equiv \frac{\pi}{2}$ $[\pi]$ and $\Re e(Y) = 0$). Thus $\Theta$ is coercive.
\end{proof} 
\subsection{Injectivity}
We prove the following "injectivity" statement in order to apply the first Fredholm Theorem:
\begin{lemma}
\label{le:injectivity} If $\beta_v \in \mathscr{B}_v$ or $\beta_v = 0$, then the sesquilinear form $A$ defined in (\ref{eq:FV}) verifies:
\[\forall u\in V(\Omega), \, (\forall v\in V(\Omega), A(u,v) = 0)\Longrightarrow (u = 0).\]
\end{lemma}
\begin{proof} Let $p\in V(\Omega)$ such that $\forall q\in V(\Omega), A(p,q) = 0$. In particular, $A(p,p) = 0$. But,
\begin{align*} A(p,p) &= \|\nabla p\|_{L^2(\Omega)}^2 - \|\mathcal{D}p\|_{L^2(\Omega)}^2 + ik\|\Tr_{\Gamma_{out}}p\|_{L^2(\Gamma_{out},\mu)}^2 \\
&+ iY\frac{Z_0}{k_0}\|\Tr_\Gamma(\mathcal{D}_1p)\|_{\chi}^2 - iY\frac{Z_0}{k_0}K^2\|\Tr_\Gamma p\|_{\chi}^2.
\end{align*}
Hence,

\begin{align}
\label{Michelle}
\begin{split}
    \Im m(A(p,p)) &= k\|\Tr_{\Gamma_{out}}p\|_{L^2(\Gamma_{out},\mu)}^2 + \Re e(Y)\frac{Z_0}{k_0}\|\Tr_\Gamma(\mathcal{D}_1p)\|_{\chi}^2 \\
    & - \frac{Z_0}{k_0}\left(\Re e(Y)\Re e(K^2) - \Im m(Y)\Im m(K^2)\right)\|\Tr_\Gamma p\|_{\chi}^2.
\end{split}
\end{align}
By transforming the expression of $K^2$:
\[K^2=k_0^2\frac{\beta_v^2}{4(1-\beta_v)} = k_0^2\frac{\beta_v^2(1-\overline{\beta_v})}{4|1-\beta_v|^2} = k_0^2\frac{\beta_v^2-\beta_v|\beta_v|^2}{4|1-\beta_v|^2}.\]
Let $\beta_v = \beta_R + i\beta_I$. Then,
\[K^2 = \frac{k_0^2}{4|1-\beta_v|^2}\left(\beta_R^2-\beta_I^2-\beta_R|\beta_v|^2 + i\beta_I(2\beta_R-|\beta_v|^2)\right).\]
Thus,
\[
    \Re e(K^2) = \frac{k_0^2}{4|1-\beta_v|^2}(\beta_R^2-\beta_I^2-\beta_R|\beta_v|^2)\hspace{8pt};\hspace{8pt} \Im m(K^2) = \frac{k_0^2}{4|1-\beta_v|^2}\beta_I(2\beta_R-|\beta_v|^2).
\]
We will subsequently show that if all terms of $\Im m(A(p,p))$ are non-negative, then all terms that appear in $\Im m(A(p,p))$ have to be null.

Recalling that $\Re e(Y) > 0$  since $\displaystyle Y = \frac 1Z$ and $\Re e(Z) > 0$,  we assume that $\beta_v$ verifies:
$$\quad- \frac{Z_0}{k_0}(\Re e(Y)\Re e(K^2) - \Im m(Y)\Im m(K^2))\|\Tr_\Gamma p\|_{\chi}^2 \geq 0.$$
To conclude that $\|\Tr_\Gamma p\|_{\chi} = 0$, we suppose that condition \eqref{betavcondition} is verified, i.e.: $$\Re e(Y)\Re e(K^2) - \Im m(Y)\Im m(K^2) < 0.$$
Solving this equation numerically yields these graphs of the admissible values of $\beta_v$ in Figure~\ref{fig:fish1} and Figure~\ref{fig:fish2}, for different values of $\displaystyle \frac{\Im m(Y)}{\Re e(Y)}$.
\begin{figure}[!htb]
    \centering
        \includegraphics[width=0.49\textwidth]{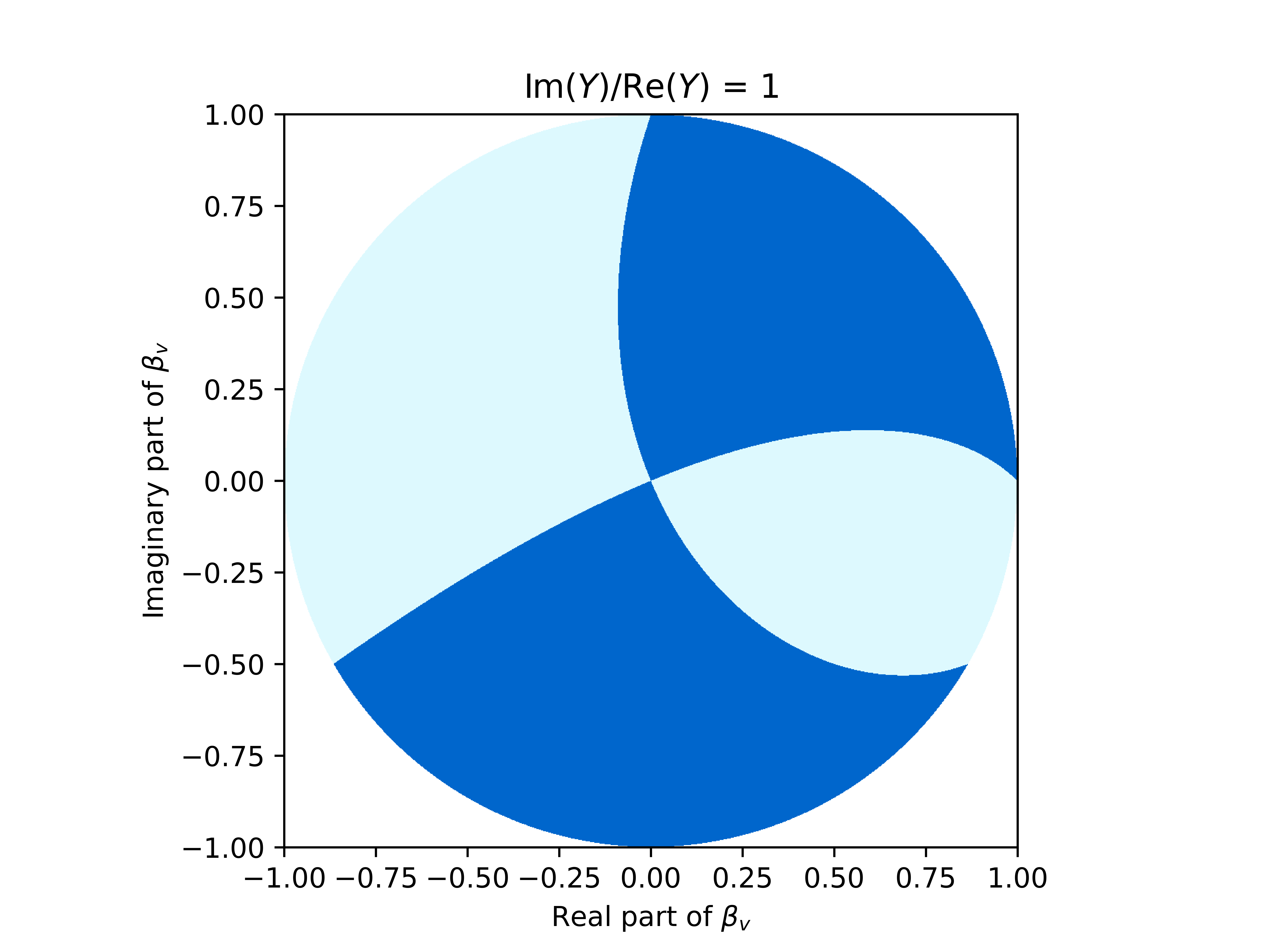}
        \includegraphics[width=0.49\textwidth]{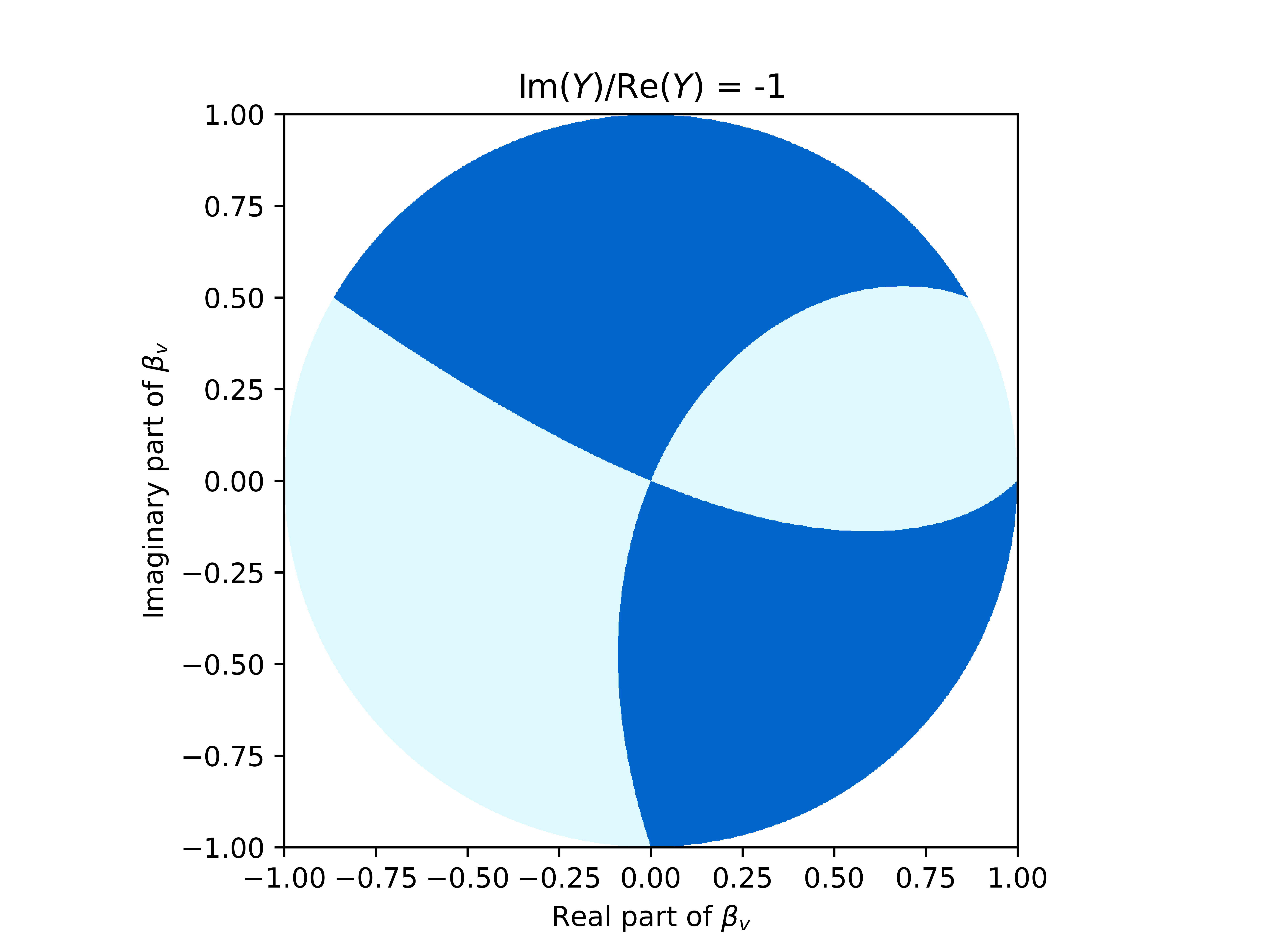}
    \caption{Admissible zones (dark blue) for $\beta_v$ for different values of $\displaystyle r= \frac{\Im m(Y)}{\Re e(Y)}$: on the left for $r=1$ and  $r=-1$ on the right.}
    \label{fig:fish1}
\end{figure}
\begin{figure}[!htb]
    \centering
        \includegraphics[width=0.45\textwidth]{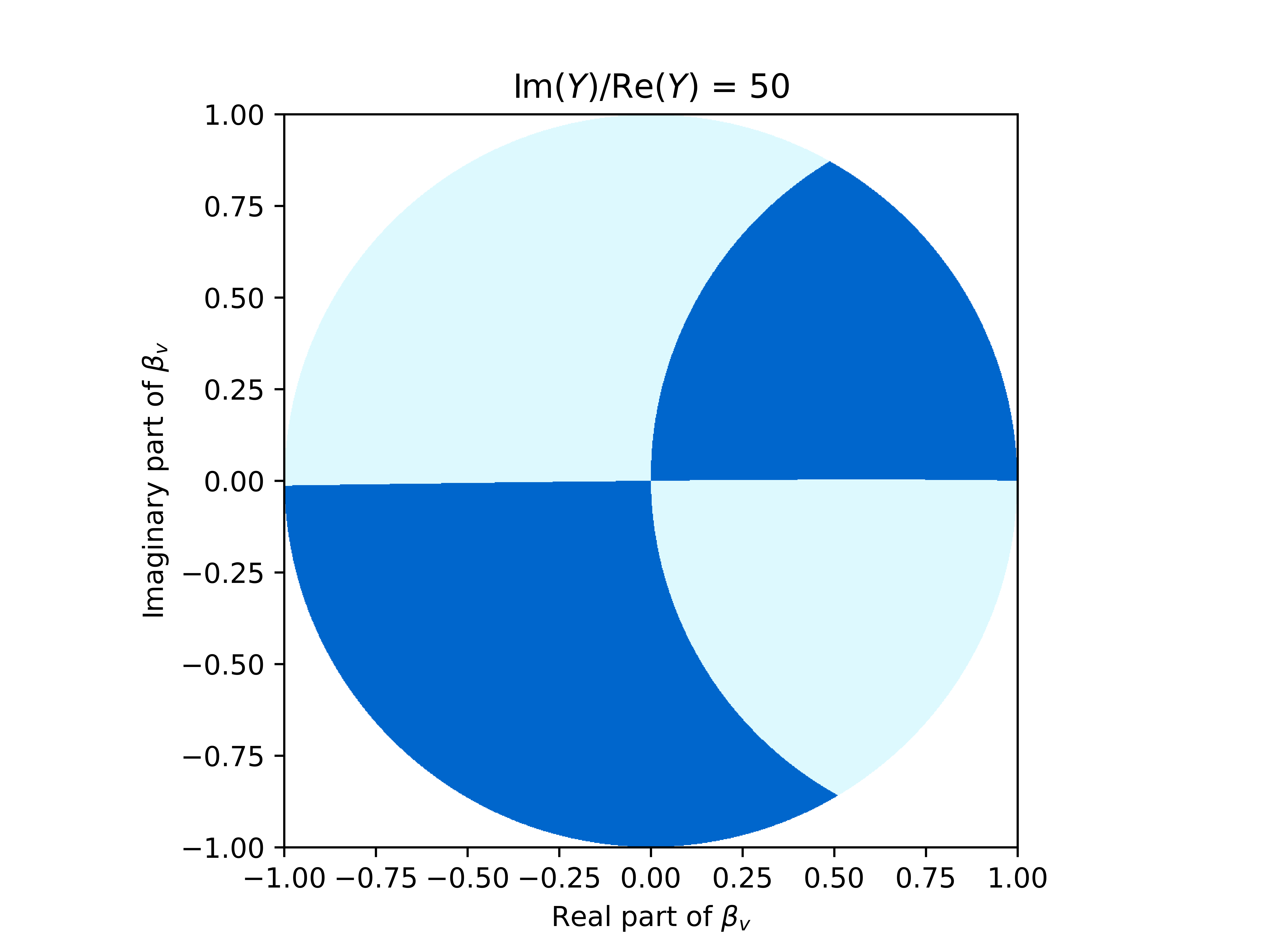}
        \includegraphics[width=0.45\textwidth]{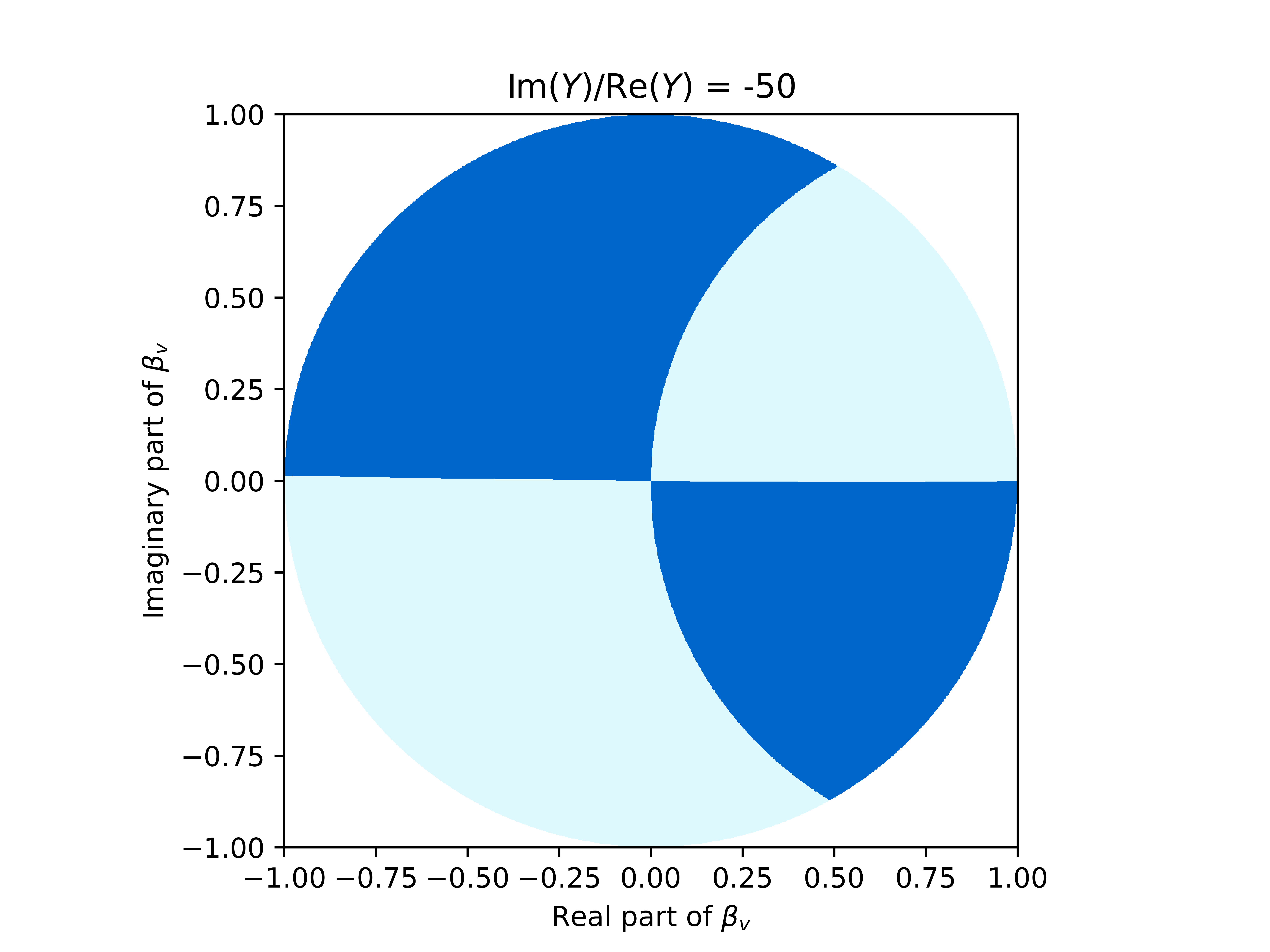}
    \caption{Admissible zones (dark blue) for $\beta_v$ for greater values of $r=\displaystyle \frac{\Im m(Y)}{\Re e(Y)}$ to compare to Fig.~\ref{fig:fish1}: on the left for $r=50$ and $r=-50$ on the right.}
    \label{fig:fish2}
\end{figure}
The admissible zone has a wing-like structure. We prove in~Appendix~\ref{AppLimitGr}  that when the ratio $\displaystyle \frac{\Im m(Y)}{\Re e(Y)}$ goes to $\pm \infty$, the graph of the admissible zone $\mathscr{B}_v$, converges to a specific geometry, that does not depend on the parameters, as shown in Figure~\ref{fig:two_images}.
\begin{figure}[!htb]
  \centering
    \includegraphics[width=0.45\textwidth]{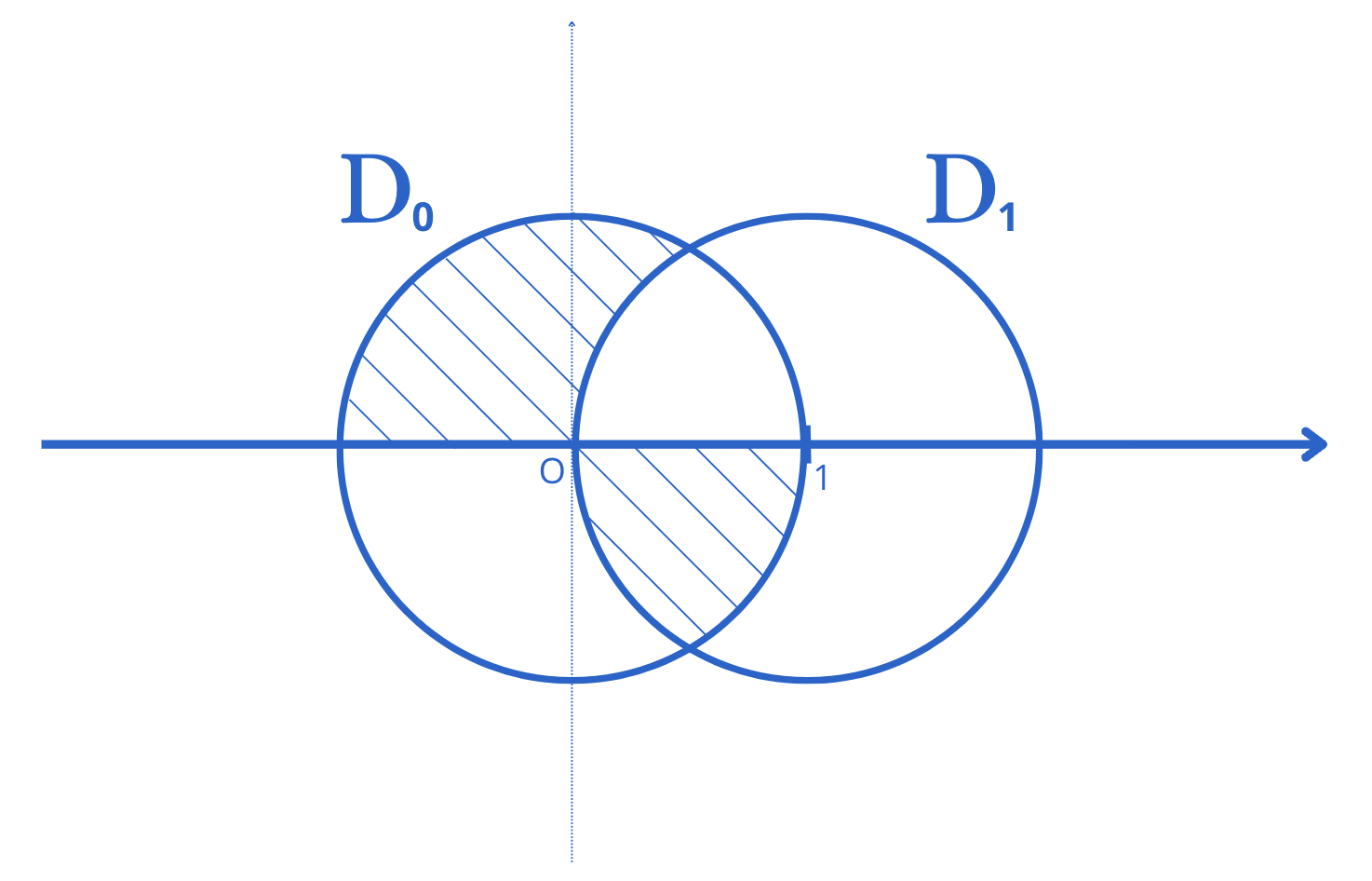} 
    \includegraphics[width=0.45\textwidth]{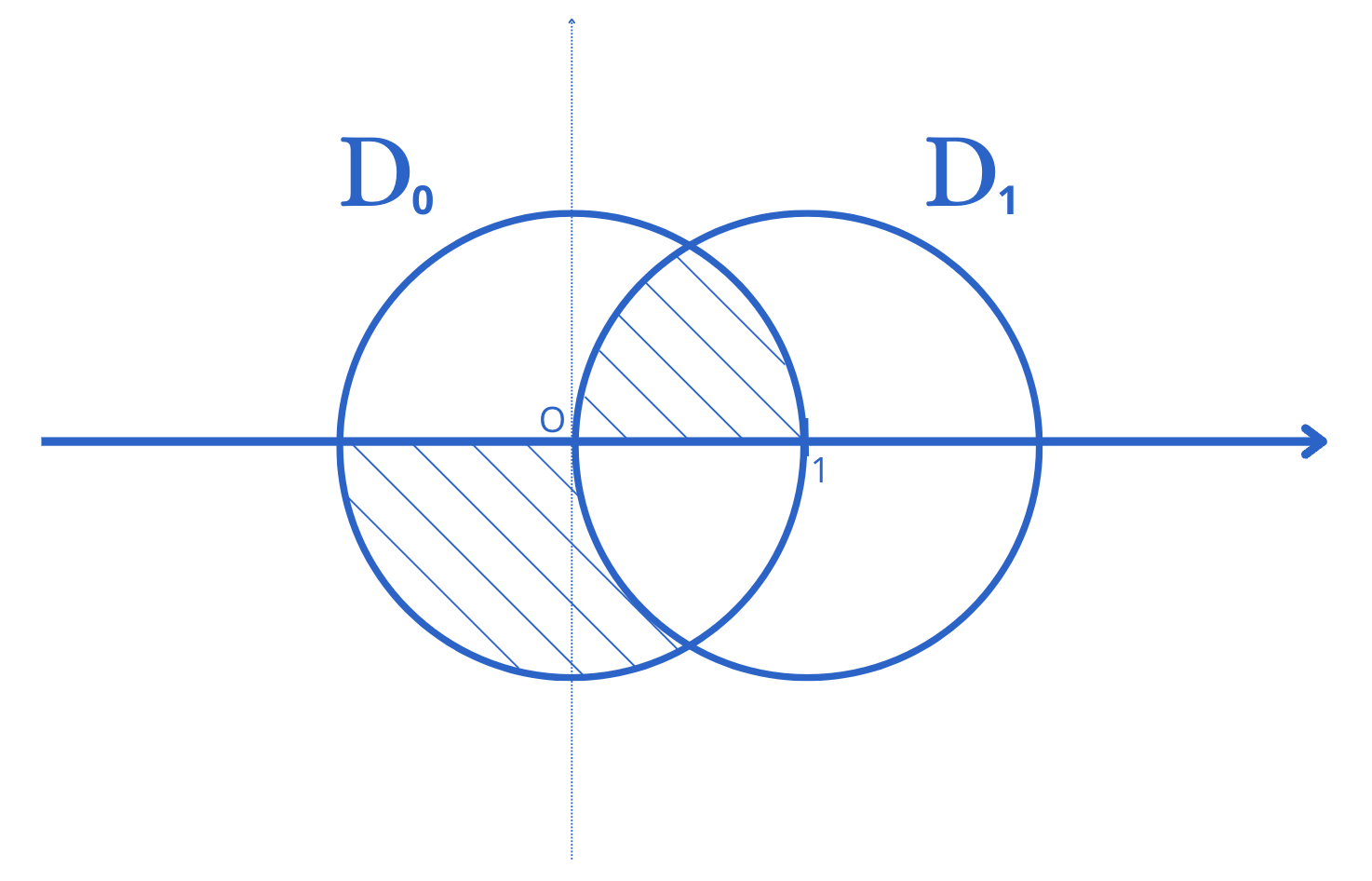} 
  \caption{Limit figures of $\mathscr{B}_v$ when ratio $r$ diverges, on the left, to $-\infty$, on the right, to $+\infty$. $D_0$ is the disc centered at $0 = 0+0i$ of radius $1$, and $D_1$ the disc centered at $1=1+0i$ of radius $1$.}
  \label{fig:two_images}
\end{figure}
In the following figure,

The graphs presented in Figure~\ref{fig:two_images} provide admissible values for \( \beta_v \), when \( |\Im m(Y)| \gg |\Re e(Y)| \), regardless of the specific values of the problem's parameters. 
\label{sec:proof well-posedness}

If $\beta_v \in \mathscr{B}_v$, then from \eqref{Michelle} we deduce $\|\Tr_\Gamma p\|_{\chi} = 0$ and $ \|\Tr_\Gamma(\mathcal{D}_1 p)\|_{\chi}= 0$. 
This implies that:
\begin{equation}
\label{tracenull}
\Tr_{\Gamma_{out}}p = 0; \quad \Tr_\Gamma(\mathcal{D}_1p)_{|\supp \chi} = 0; \quad \Tr_\Gamma p_{|\supp \chi} = 0.
\end{equation}
Thus $\left.\Tr_\Gamma(\partial_xp)\right|_{\supp \chi} = 0$.

If $\beta_v = 0$, then from \eqref{Michelle} we obtain $\Tr_{\Gamma_{out}}p = 0$ and $\Tr_\Gamma(\mathcal{D}_1p)_{|\supp \chi} = 0$. Regardless of whether $\beta_v \in \mathscr{B}_v$ or $\beta_v = 0$, $A(p,q)$ becomes: 
$$ A(p,q) = (\nabla p,\nabla q)_{(L^2(\Omega))^3} - (\mathcal{D}p,\mathcal{D}q)_{L^2(\Omega)} = 0. $$
Thus $p$ is a solution of the differential equation:
\[
    \begin{cases}
    \displaystyle \Delta  p  + \mathcal{D}^2 p = 0 &\text{in }\Omega,
    \\
    \displaystyle \frac{\partial p}{\partial n} = 0 &\text{on }\Gamma,
    \\
    \displaystyle \Tr_{\Gamma_{in}}p = 0 &\text{on }\Gamma_{in},
    \\
    \displaystyle \Tr_{\Gamma_{out}}p = 0 &\text{on }\Gamma_{out}.
    \end{cases}
\]
We define the extension:
\[
w = \begin{cases}
    p &\text{on }\Omega,\\
    0 &\text{on } \Omega_\infty \setminus \Omega,
\end{cases}
\]
for $w\in \{h\in H^1(\Omega_\infty), \Delta h\in L^2(\Omega_\infty)\}$ where $\Omega_\infty = \{(x,y,z)\in \bR^3, y^2+z^2 < R^2\}$ is the infinite cylinder extending $\Omega$, and its boundary $\partial \Omega = \Gamma_\infty$. Then, we write $\Omega_\infty = \bR\times D$ where $D = \mathcal{B}(0_{\bR^2},R)$ is the open disk in $\bR^2$ centered at $(0,0)$ with radius $R$.

Therefore, $w$ satisfies the differential system:
\[
    \begin{cases}
    \displaystyle \Delta  w  + \mathcal{D}^2 w = 0 &\text{on }\Omega_\infty, 
    \\
    \displaystyle \frac{\partial w}{\partial n} = 0 &\text{on }\Gamma_\infty.
    \end{cases}
\]
Finally, the transverse Fourier transform of $w$ is defined almost everywhere as:
\[\forall (\xi,y,z)\in \Omega_\infty, \hat{w}(\xi,y,z) = \int_\bR w(x,y,z)e^{-i\xi x}dx.\]
Noticing that the transverse Fourier transform of the operator $\partial_x$ is $-i\xi$, and that of the operator $\mathcal{D}$ is $\hat{\mathcal{D}} = k_0 - M_0\xi$, $\hat{w}$ satisfies the following for all $\xi \in \mathbb{R}$:
\[
    \begin{cases}
    \displaystyle -\frac{\partial^2 \hat{w}}{\partial y^2}-\frac{\partial^2 \hat{w}}{\partial z^2} = \left(-\xi^2 + (k_0-M_0\xi)^2\right)\hat{w} &\text{on }\{\xi\}\times D, 
    \\
    \displaystyle \frac{\partial \hat{w}}{\partial n} = 0 &\text{on }\{\xi\}\times \partial D. 
    \end{cases}
\]
However, the eigenvalue problem of the Laplacian on $D$ (bounded domain) with $\lambda \geq 0$ and Neumann boundary conditions,
\[    
    \begin{cases}
        \displaystyle -\Delta h = \lambda^2 h &\text{on }D,
        \\
        \displaystyle \frac{\partial h}{\partial n} = 0&\text{on }\partial D,
    \end{cases}
\]
has for unique solution $h=0$ except for a countable number of values $\lambda^2 \geq 0$ (which are related to the eigenvalues of the Neumann Laplacian). We thus gather all $\xi\in \bR$ related to those values into the countable set $\mathscr A\subset \bR$. Therefore, \[\{(\xi,y,z) \in \Omega_\infty, \hat{w}(\xi,y,z) \neq 0\} \subseteq \left(\mathscr A\times D\right) \cup \left\{(\xi,y,z)\in \mathscr A^c \times D, \hat{w}(\xi,y,z) \neq 0\right\}.\]
Consequently, 
\[
\lambda^{(3)}(\{(\xi,y,z) \in \Omega_\infty, \hat{w}(\xi,y,z) \neq 0\})
\]
\[
\leq \underbrace{\lambda\left(\mathscr A\right)}_{= \hspace{2pt}0}\times \lambda^{(2)}(D) + \int_{\displaystyle \mathscr A^c} \underbrace{\lambda^{(2)}(\{(y,z) \in D, \hat{w}(\xi,y,z) \neq 0\})}_{= \hspace{2pt}0}d\xi.
\]
$\lambda^{(3)}(\{(\xi,y,z) \in \Omega_\infty, \hat{w}(\xi,y,z) \neq 0\} = 0$ and $\hat{w} = 0$ almost everywhere, thus in $L^2(\Omega_\infty)$. Hence $p = 0$, which completes the proof of injectivity for these values of $\beta_v$.

\subsection{Continuous dependence}

Let us prove the continuous dependence \eqref{estimatebound} of the weak solution on the source terms $f\in L^2(\Omega)$ and $\eta \in L^2(\Gamma,\mu)$.

Let us first remark, in accordance with the Fredholm-type decomposition already established, the existence of an inner product $[\cdot,\cdot]_{V(\Omega)}$ equivalent to $\langle \cdot,\cdot\rangle_{V(\Omega)}$, $c\neq 0$ and $K':V(\Omega)\to V(\Omega)$ compact such that
\begin{equation}
\label{repchange}
    \forall p,q\in V(\Omega), \; A(p,q) = [(cId-K')p,q]_{V(\Omega)}.
\end{equation}
Having proved from the previous subsection that $cId-K'$ is bijective and continuous, it is therefore a homeomorphism by the Banach-Schauder theorem, and we set $T = (cId-K')^{-1}\in \mathcal L(V).$

Denote by $W:L^2(\Omega)\oplus L^2(\Gamma,\mu)\to V(\Omega)$ the operator that associates $(f,\eta)$ to the solution of the variational formulation~\eqref{eq:FV} of~\eqref{PH} for $(f,\eta)$, and $S:L^2(\Omega)\oplus L^2(\Gamma,\mu)\to V(\Omega)$ the operator given by the Riesz representation theorem such that
\begin{equation}
\label{Soperator}
    \forall (f,\eta) \in L^2(\Omega)\oplus L^2(\Gamma,\mu),\forall q\in V(\Omega), \; (\eta,\Tr_\Gamma q)_{L^2(\Gamma,\mu)}-(f, q)_{L^2(\Omega)} = [S(f,\eta), q]_{V(\Omega)}.
\end{equation}
It is easy to see that $S$ is linear. For continuity, it suffices to notice that for $(f,\eta) \in L^2(\Omega)\oplus L^2(\Gamma,\mu)$, using Cauchy-Schwarz and Poincaré inequalities ($C_P$ denotes the Poincaré constant) as well as continuity of $\Tr_\Gamma: V(\Omega)\to L^2(\Gamma,\mu)$:
\begin{align*}
[S(f,\eta),S(f,\eta)]_{V(\Omega)} &= (\eta,\Tr_\Gamma (S(f,\eta)))_{L^2(\Gamma,\mu)}-(f, S(f,\eta))_{L^2(\Omega)} \\ &\leq (C_P+\|\Tr_\Gamma\|_{\mathcal L(V(\Omega),L^2(\Gamma,\mu))})(\|f\|_{L^2(\Omega)}+\|\eta\|_{L^2(\Gamma,\mu)})\|S(f,\eta)\|_{V(\Omega)}.
\end{align*}
Thus $S$ is continuous by equivalence of inner products.

From the variational formulation \eqref{eq:FV}, along with~\eqref{repchange} and~\eqref{Soperator}, we deduce that $(f,\eta) \in L^2(\Omega)\oplus L^2(\Gamma,\mu)$ verifies
$$(cId-K')W(f,\eta) = S(f,\eta).$$
Thus $W = TS$, and setting $\hat C= \|T\|_{\mathcal L(V(\Omega))}\|S\|_{\mathcal L(L^2(\Omega)\oplus L^2(\Gamma,\mu),V(\Omega))} > 0$, a constant depending only on $d$, $c_d$, $R$ and $L$ from~\eqref{EqOmega}, $\chi$, $\beta_v$ and other physical constants from~\eqref{EqPhysConstants}, we obtain the promised result:
\begin{equation}
    \forall (f,\eta) \in L^2(\Omega)\oplus L^2(\Gamma,\mu), \|W(f,\eta)\|_{V(\Omega)}\leq \hat C (\|f\|_{L^2(\Omega)}+\|\eta\|_{L^2(\Gamma,\mu)}).
\end{equation}
\end{proof}

\section{Parametric shape optimization of liner distribution}
\label{sec:shapeoptimization}
As the direct problem~\eqref{PH} is weakly well-posed, we consider the optimal control problem of minimization of its energy in the framework of the parametric shape optimization on the boundary~$\Gamma$. 

Let $\chi\in L^\infty(\Gamma,\mu)$ be the characteristic function of the distribution of the liner on $\Gamma$:
\begin{equation}\label{EqChi}
\forall x \in \Gamma, \quad \chi(x) = \begin{cases}
1 & \text{if there is a liner in } x,\\
0 & \text{if there is no liner in } x,
\end{cases}
\end{equation}
 having a fixed $L^1 (\Gamma,\mu)$-norm, consisting in the volume fraction of the liner on $\Gamma$:
 \begin{equation}
\label{quantitygamma}
   0< \gamma := \|\chi\|_{L^1(\Gamma,\mu)}=\int_\Gamma \chi d\mu<\mu(\Gamma).
\end{equation}
We exclude two limit cases $\gamma=0$ and $\gamma=\mu(\Gamma)$ and fix a value $\gamma\in (0,\mu(\Gamma))$. Therefore, 
we define the class of admissible liner distributions:
\begin{equation}\label{EqUad}
	U_{ad}(\gamma) := \left\{\chi \in L^{\infty}(\Gamma,\mu)  \middle| \, \mu\text{-a.e } x \in \Gamma, \chi(x) \in \{0,1\}, 0 < \gamma = \int_{\Gamma} \chi \, d\mu < \mu(\Gamma)\right\}.
\end{equation}
Let us now consider the total acoustical energy of problem~\eqref{PH} which we want to minimize on $U_{ad}(\gamma)$, first for a fixed wave number $k_0>0$ and then for all bounded wavenumber integral $I\subset \R^+$. 
We emphasize that different wave numbers $k_0$ and liner distributions $\chi$ generally correspond to different solutions $p$ of~\eqref{PH} and vary the energy.
As in~\cite{MAGOULES-2025}, we define the following general energy functional $J(k_0,\chi): I\times U_{ad}(\gamma)\to \R$ by
\begin{equation}\label{EqJ}
	 J(k_0,\chi)=a\int_\Omega |u(k_0,\chi)|^2\dx+b\int_\Omega |\nabla u(k_0,\chi)|^2 \dx+d\int_\Gamma |\operatorname{Tr} u(k_0,\chi)|^2 d\mu
\end{equation}
with positive constants $a\ge 0$, $b\ge 0$ and $d\ge0$, $a^2+b^2>0$. If $a\ge0$, and  $b$ with $d$ are strictly positive, the expression of $J$ defines an equivalent norm on $H^1(\Omega)$, and hence, on $V(\Omega)$.  
Therefore, our final aim is to minimize the ``total'' energy on $U_{ad}(\gamma)$:
\begin{equation}\label{EqTotalJ}
	\hat{J}(\chi):=\int_I J(k_0,\chi) d k_0, \quad \min_{\chi\in U_{ad}(\gamma)}\hat{J}(\chi).
\end{equation}

Thus, we formulate two optimization problems:
\begin{definition}\label{DefPOP}\textbf{(Parametric optimization problems)}
	In the assumptions of Theorem~\ref{th:well-posedness} for a fixed $\gamma\in (0,\mu(\Gamma))$, and the source of the noise $f$ 
	\begin{enumerate}
		\item \textbf{for a fixed wavenumber} $k_0>0$, to find $\chi_{opt}\in U_{ad}(\gamma)$ for which there exists the (unique) solution $u(k_0,\chi_{opt})\in V(\Omega)$ of the convected Helmholtz problem with the generalized Myers boundary condition~\eqref{PH} considered with $\chi=\chi_{opt}$, such that
$$J(k_0,\chi_{opt})=\min_{\chi\in U_{ad}(\gamma)}J(k_0,\chi).$$
\item \textbf{for a bounded range of  wavenumbers} $I$, to find $\chi_{opt}\in U_{ad}(\gamma)$ for which there exists for all $k_0\in I$ the (unique) solution $u(k_0,\chi_{opt})\in V(\Omega)$ of  problem~\eqref{PH} considered with $\chi=\chi_{opt}$, such that
$$\hat{J}(\chi_{opt})=\min_{\chi\in U_{ad}(\gamma)}\int_I J(k_0,\chi)d k_0.$$
	\end{enumerate}
\end{definition}


%

\subsection{Relaxation method}
By its definition, as it was also mentioned in~\cite[Sec.~3]{MAGOULES-2025}, the set of the admissible shapes $U_{ad}(\beta)$ is not closed for the weak$^*$ convergence of $L^\infty(\Gamma,\mu)$~\cite{HENROT-2005}:
if a sequence of characteristic functions $(\chi_n)_{n\in \N}$ converges weakly$^*$ in $L^\infty(\Gamma,\mu)$ to a function $h\in L^\infty(\Gamma,\mu)$, it does not follows  that the weak$^*$ limit function $h$ is a characteristic function, $i.e.$ takes only two values $0$ and $1$. Hence, \(U_{ad}(\gamma)\) is not weakly$^*$ compact. To address this issue, we follow the standard relaxation approach~\cite[p.277]{HENROT-2005}, consisting in introducing the (convex) closure of \(U_{ad}(\gamma)\) in the weakly$^*$ topology of $L^\infty(\Gamma,\mu)$: 
\begin{equation}
    \label{Uadstar}
    U_{ad}^*(\gamma) := \left\{\chi \in L^{\infty}(\Gamma,\mu) \, \middle| \,0\leq \chi \leq 1 \hspace{2pt} \mu\text{-a.e},\hspace{2pt} 0 < \gamma = \int_{\Gamma} \chi \, d\mu < \mu(\Gamma)\right\}.
\end{equation} 
Let us for simplicity normalize the values of $\mu$ on $\Gamma$ and suppose in what follows that $\mu(\Gamma)=1$. This makes of $\gamma$ the percentage rate of the liner on $\Gamma$, $0<\gamma<1$.
We notice that  $\|\chi\|_{L^\infty(\Gamma,\mu)}=1$ for all $\chi\in U_{ad}(\gamma)$, while 
for all $\chi\in U_{ad}^*(\gamma)$ it holds 
\begin{equation}\label{EqBLInfNorm}
	0<\gamma\le \| \chi\|_{L^\infty(\Gamma,\mu)}\le 1.
\end{equation}
By~\cite[Theorem~3.2]{MAGOULES-2025} and~\cite[Proposition 7.2.14]{HENROT-2005},  $U_{ad}^*(\gamma)$ is the weak$^*$ closed convex hull of $U_{ad}(\gamma)$ and $U_{ad}(\gamma)$ is exactly the set of extreme points of the convex set $U^*_{ad}(\gamma)$. 

We denote by $J^*$ the natural extension of $J$ on the relaxed space $U_{ad}^*(\gamma)$: 
\begin{align}
\label{EqJ*}
\begin{split}
	\forall \chi\in U^*_{ad}(\gamma), \quad J^*(k_0,\chi)&=a\int_\Omega| u(k_0,\chi)|^2\dx+b\int_\Omega |\nabla u(k_0,\chi)|^2 \dx\\
    &+d\int_\Gamma |\operatorname{Tr} u(k_0,\chi)|^2 d\mu,
\end{split}
\end{align}
which in addition satisfies $J^*(k_0,\chi)|_{U_{ad}(\gamma)}=J(k_0,\chi)$. 
Here, $u(k_0,\chi)$ is the weak solution of system~\eqref{PH} found for a chosen $(k_0,\chi)$.
We also denote
\begin{equation}\label{EqhatJ*}
	\forall \chi\in U^*_{ad}(\gamma) \quad \hat{J}^*(\chi)=\int_I J^*(k_0,\chi) d k_0,
\end{equation}
satisfying $\hat{J}^*(\chi)|_{U_{ad}(\gamma)}=\hat{J}(\chi)$.

To solve the parametric optimization problem on $U^*_{ad}(\gamma)$ we need to ensure that the constant $\hat{C}$ in estimate~\eqref{estimatebound} does not depend on $\chi$, when $\chi\in U^*_{ad}(\gamma)$. 
As $\mu(\Gamma_{in})>0$, then it follows, as explained in~\cite{MAGOULES-2025}, from the upper uniform boundedness of the $L^\infty$ norm of all $\chi$ on $U^*_{ad}(\gamma)$ (see~\eqref{EqBLInfNorm}) and the equivalence of norms with uniform on $\chi$ constants:   for all $\chi \in U^*_{ad}(\gamma)$ 
there exist $C_0>0$ independent on $\chi \in U^*_{ad}(\gamma)$ such that
\begin{equation}\label{EqNormEq}
	\forall v\in V(\Omega) \quad \|v\|_{H^1_0(\Omega)}\le \|v\|_{V(\Omega),\chi}\le C_0 \|v\|_{H^1_0(\Omega)}.
\end{equation}
To prove it, we use the continuity of the trace operator $\Tr_\Gamma\in \mathcal{L}(V(\Omega),L^2(\Gamma,\mu))$ and the differential operator  $\mathcal{D}_1\in \mathcal{L}(H^1(\Omega),L^2(\Omega))$ (see~\eqref{D1andK} for definition), and the Poincaré inequality on the cylindrical domain $\Omega$ to obtain $$C_0=1+C(k_0,\beta_v, M_0) C(\|\mathrm{Tr}_\Gamma\|_{\mathcal{L}(V(\Omega),L^2(\Gamma,\mu))}, \|\mathcal{D}_1\|_{\mathcal{L}(H^1(\Omega),L^2(\Omega))})C_{P}(\Omega,\mu),$$ independent on $\chi$. Here, by $C_P$ is denoted the Poincaré constant. 

\begin{lemma}\label{LemIndepChi}
	Let $\gamma\in ]0,1[$ (for $\mu(\Gamma)=1$) be fixed and all assumptions of Theorem~\ref{th:well-posedness} hold. Then for all $\chi\in U^*_{ad}(\gamma)$, there exists a constant $\hat{C}^*>0$, depending only on  $k_0$, $\beta_v$, $M_0$ and on $C_P$ (the Poincaré uniform constant depending only on $L$, $R$ (see~\eqref{EqOmega}), $d$ and $A$ (see~\eqref{defmu})), \textit{but not on $\chi$}, such that estimate~\eqref{estimatebound} holds for the corresponding weak solution of~\eqref{PH}. %
\end{lemma}

Therefore, the minimization problem becomes: 
\begin{equation}
\label{minimization *}
    \displaystyle J^*(k_0,\chi^*) = \min_{\chi \in U_{ad}^*(\gamma)}J^*(k_0,\chi) \hbox{ and } \hat{J}^*(\chi^*) = \min_{\chi \in U_{ad}^*(\gamma)}\hat{J}^*(\chi).
\end{equation}
First we show the weak$^*$ continuous dependence of the solution and the energy on the liner distribution $\chi$ for a fixed wavenumber $k_0>0$. As $k_0$ is fixed, we simplify the notations by omitting $k_0$ and instead of
	$p(f_0,\chi)$ and  $J^*(f_0,\chi)$  are denoted by $p(\chi)$  and  $J^*(\chi,u(\chi))$  respectively.
\begin{proposition}[Continuity on $\chi$]
Let $\chi\in U^*_{ad}(\gamma)$ for a fixed $\gamma\in (0,1)$, the assumptions of Theorem~\ref{th:well-posedness} hold 
and $p(\chi)\in V(\Omega)$ be the weak solution of the variational formulation~\eqref{eq:FV}. 
\begin{enumerate}
	\item[(i)] The mapping $\chi \longmapsto p(\chi)$ is a continuous and compact operator from $ U^*_{ad}(\gamma)$ to $V(\Omega)$,
	\item[(ii)] The functional $J^*$ is continuous on $U_{ad}^*(\gamma)$ endowed with the weak$^*$ topology.
\end{enumerate}
\end{proposition}
\begin{proof}
    Let us prove point (i), then point (ii) will follow immediately. Let $\chi_m$ be a sequence that converges weakly$^*$ to $\chi$ in $L^\infty(\Gamma,\mu)$, with $\chi \geq 0$ and for all $m \in \mathbb{N}, \chi_m \geq 0$. Let $p_m$ be the solution of \eqref{PH} for $\chi_m$ and $p$ for $\chi$. Then $v_m = p_m - p $ is a solution of:
\begin{equation*}
    (P_{H,m}) \hspace{2pt} : \hspace{2pt}
    \begin{cases}
    \displaystyle \Delta  v_m  + \mathcal{D}^2 v_m = 0 \in L^2(\Omega),
    \\
    \displaystyle \frac{\partial v_m}{\partial n} + iY\frac{Z_0}{k_0}\chi \Tr_\Gamma \Bigl[ \mathcal{D}(\mathcal{D} + iM_0\beta_v\partial_x) v_m \Bigr] = \eta_m,
    \\
    \displaystyle \Tr_{\Gamma_{in}}v_m = 0,
    \\
    \displaystyle \frac{\partial v_m}{\partial n} + ik \Tr_{\Gamma_{out}}v_m = 0.
    \end{cases} \label{PHm}
\end{equation*}
where $\eta_m = \displaystyle - ik_0YZ_0(\chi_m - \chi)\Tr_\Gamma \Bigl[ \mathcal{D}(\mathcal{D} + iM_0\beta_v\partial_x) p_m\Bigr]$.

This problem is well-posed according to Theorem~\ref{th:well-posedness}. Since $d$, $c_d$, $R$ and $L$ from~\eqref{EqOmega}, $\chi$, $\beta_v$ and other physical constants from~\eqref{EqPhysConstants} do not depend on $m$, we have the existence of an uniform on $m$ constant $C>0$ such that 
\begin{equation*}
\label{eq:Ineg}
 \forall m\in \mathbb{N},   \|v_m\|_{V(\Omega),\chi}\leq C\|\eta_m\|_{L^2(\Gamma,\mu)}.
\end{equation*}
Furthermore, without loss of generality (otherwise switch to the equivalent inner product), we have in accordance with~\eqref{repchange} that
\begin{equation}
\label{eq:PS}
   \forall p,q\in V(\Omega), \quad \langle (cId - K')p,q \rangle_{V(\Omega),\chi} = A(p,q);
\end{equation}
where $c \in \mathbb{C}\setminus \{0\}$ is a constant and $K':V(\Omega)\to V(\Omega)$ a compact operator.
Consequently,
\begin{equation}
    \label{eq:hmvm}
    \forall q\in V(\Omega), \, \forall m\in \mathbb{N}, \quad \langle (cId - K')v_m,q\rangle_{V(\Omega),\chi} = A(v_m,q)=-(\eta_m,\Tr_\Gamma q)_{L^2(\Gamma,\mu)}.
\end{equation}

Firstly,
since $\chi_m \stackrel{\ast}{\rightharpoonup} \chi$ in $L^\infty(\Gamma,\mu)$, the sequence $(\chi_m)_{m \in \mathbb{N}}$ is bounded in $L^\infty(\Gamma,\mu)$ and thus the same is true for $\displaystyle(\chi - \chi_m)_{m \in \mathbb{N}}$. Moreover,
\( p \) is the weak solution of ($P_H$) associated with the function \( \chi \), thus it belongs to \( V(\Omega) \) and its trace on \( \Gamma \) is well defined and naturally belongs to \( L^2(\Gamma,\mu) \). Furthermore, the norm of the trace of \( p \) on \( \Gamma \) does not depend on \( m \).

Thus, $(v_m)_{m \in \mathbb{N}}$ is bounded in $V(\Omega)$, which is a Hilbert space. Therefore, there exists a subsequence that converges weakly:
\begin{center}
    $\exists(m_j)_{j \in \mathbb{N}} \subset \mathbb{N}$ increasing s.t. $v_{m_j} \rightharpoonup v$ in $V(\Omega)$ with $v \in V(\Omega)$.
\end{center}
We will now show that $v = 0$. According to \eqref{eq:hmvm}:
\[\forall q\in V(\Omega), \, \forall j\in \mathbb{N}, \quad \langle (cId - K')v_{m_j},q \rangle_{V(\Omega),\chi} = -(\eta_{m_j},\Tr_\Gamma q)_{L^2(\Gamma,\mu)}.\]
Taking the limit, 
\[\forall q\in V(\Omega), \quad \langle (cId - K')v,q\rangle_{V(\Omega),\chi} = 0, \]
by the uniqueness of the weak limit, and because $\eta_{m_j}\rightharpoonup 0$.
Since the operator $cId - K'$ is bijective according to Fredholm's theorem, and taking $q =(cId - K')v$, we conclude that $v=0$.

Thus, we have shown that $0$ is the only weak accumulation point of the sequence $(v_m)$. Therefore, $v_m\rightharpoonup 0$.
Next, using \eqref{eq:hmvm} once again with $p=q=v_m$, it follows that $A(v_m,v_m) \longrightarrow 0$.
Indeed, the Cauchy-Schwarz inequality allows us to bound this term:
$$|(\eta_m,\Tr_\Gamma v_m)_{L^2(\Gamma,\mu)}| \leq \|\eta_m\|_{L^2(\Gamma,\mu)}\|\Tr_\Gamma v_m\|_{L^2(\Gamma,\mu)}$$
and the result is immediate with the compactness of the operator $\Tr_\Gamma : V(\Omega) \longrightarrow L^2(\Gamma,\mu)$, which gives us the strong convergence of the sequence $(\Tr_\Gamma v_m)_m$.
Finally,
\[\langle (cId - K')v_m,v_m \rangle_{V(\Omega),\chi} = A(v_m,v_m) \longrightarrow 0.\]
Therefore,
\[c\|v_m\|^2_{V(\Omega),\chi} - \langle K'v_m,v_m  \rangle_{V(\Omega),\chi}  \longrightarrow 0.  
\]
Hence, since $\langle K'v_m,v_m  \rangle_{V(\Omega),\chi}\longrightarrow 0$ due to the compactness of the operator $K'$, we deduce that 
\[ \|v_m\|^2_{V(\Omega),\chi} \longrightarrow 0. 
\]
Thus $v \longrightarrow 0$, hence the continuity of the mapping $\chi \longmapsto p(\chi)$.\end{proof}

\subsection{Existence of an optimal liner distribution }

From previous results, we deduce the following theorem.
\begin{theorem}[Existence of a minimizer]\label{ThOptShape}
Let $\Omega\subset \bR^3$ be the cylindrical domain defined in \eqref{EqOmega} and all assumptions of Theorem~\ref{th:well-posedness} are satisfied for a fixed $d$-upper regular measure 
$\mu$ with $d\in (1,2]$, $\mu(\Gamma)=1$, and $\beta_v \in \mathscr{B}_v \cup \{0\}$ with $\mathscr{B}_v$ defined by \eqref{poissondef}.

 Then for fixed sources $f\in L^2(\Omega)$, $\eta \in L^2(\Gamma,\mu)$ and 
	for a given liner distribution quantity $\gamma\in ]0, 1[$,  there exists (at least one) optimal distribution $\chi^{opt}\in U^*_{ad}(\beta)$ and the corresponding optimal solution $u(f_0,\chi^{opt})\in V(\Omega)$ of system~\eqref{PH}, such that
	\begin{equation}\label{EqExistMin}
	J^*(k_0,\chi^{opt})=\min_{\chi \in U^*_{ad}(\beta)} J^*(k_0,\chi)=\inf_{\chi \in U_{ad}(\beta)} J(k_0,\chi),	
	\end{equation}
	and there exists $\hat{\chi}^{opt}\in U^*_{ad}(\beta)$ such that on a fixed bounded plage of wavenumbers $I\subset \R^{+ *}$
	\begin{equation}\label{EqExistFreqMin}   
		\hat{J}^*(\hat{\chi}^{opt})=\min_{\chi \in U^*_{ad}(\beta)} \hat{J}^*(\chi)=\inf_{\chi \in U_{ad}(\beta)} \hat{J}(\chi).		
	\end{equation}
\end{theorem}
\begin{proof}    
We consider  minimizing sequences $\displaystyle (\chi_j)_{j\in \mathbb{N}}\subset U^*_{ad}(\gamma)$ and $\displaystyle (\hat{\chi}_j)_{j\in \mathbb{N}}\subset U^*_{ad}(\gamma)$ such that $\displaystyle J^*(k_0,\chi_j) \xrightarrow{j\to +\infty} \inf_{\chi \in U^*_{ad}(\gamma)}J^*(k_0,\chi)$ and $\displaystyle \hat{J}^*(\hat{\chi}_j) \xrightarrow{j\to +\infty} \inf_{\chi \in U^*_{ad}(\gamma)}\hat{J}^*(\chi)$ respectively. As $U^*_{ad}(\gamma)$ is weakly$^*$ compact (in $L^\infty(\Gamma,\mu)$), there exist subsequences of the minimizing sequences weakly$^*$ converging in  $U^*_{ad}(\gamma)$ to $\chi^{opt},$ $\hat{\chi}^{opt} \in U^*_{ad}(\gamma)$ (and the corresponding solutions of the convected Helmholtz system with the generalized Myers boundary condition~\eqref{PH}) respectively. Let us still denote these minimizing subsequences by $\displaystyle (\chi_j)_{j\in \mathbb{N}}$ and $\displaystyle (\hat{\chi}_j)_{j\in \mathbb{N}}$ respectively.
Thanks to the weakly$^*$ continuity of $J^*$ and $\hat{J}^*$ on $U^*_{ad}(\gamma)$ (by the weakly$^*$ continuity of $p(\cdot,\chi)$ and the definitions of $J^*$ and $\hat{J}^*$), 
\begin{equation*}
J^*(k_0,\chi^{opt}) = \lim_{j\to +\infty} J^*(k_0,\chi_j) = \inf_{\chi \in U^*_{ad}(\gamma)}J^*(k_0\chi)
\end{equation*}
and
\begin{equation*}
\hat{J}^*(\hat{\chi}^{opt}) = \lim_{j\to +\infty} \hat{J}^*(\hat{\chi}_j) = \inf_{\chi \in U^*_{ad}(\gamma)}\hat{J}^*(\chi).
\end{equation*}
In other words, $\chi^{opt},$ $\hat{\chi}^{opt} \in U^*_{ad}(\gamma)$ realize the minima of $J^*$ and $\hat{J}^*$ respectively on $U^*_{ad}(\gamma)$ (by a continuity on a compact).
In addition, $$\min_{\chi \in U^*_{ad}(\gamma)} J^*(k_0,\chi)=\inf_{\chi \in U_{ad}(\gamma)} J(k_0,\chi)$$ as $U^*_{ad}(\gamma)$ is the closure of $U_{ad}(\gamma)$ and $J^*$ takes the same values as $J$ on $U_{ad}(\gamma)$ (see~\cite[Theorem~3.2]{MAGOULES-2025}). In the same way, we conclude for $\hat{J}^*$.
\end{proof} 

\appendix

\section{Variational Formulation}\label{sec:variational formula}
The objective of this part of the Appendix is to prove the variational formulation Proposition~\ref{prop:FV}:
\begin{proposition} (Variational Formulation)

The variational formulation associated with \eqref{PH} can be expressed as:
\begin{equation}
\label{eq:FVappen}
    \forall q\in V(\Omega), \ A(p,q) = l(q);
\end{equation}
where we define the following forms:
\[\forall q\in V(\Omega), \, l(q) = (\eta,\Tr_\Gamma q)_{L^2(\Gamma,\mu)} -(f,q)_{L^2(\Omega),}\]
and $\forall p,q\in V(\Omega),$
\[ A(p,q) = (\nabla p,\nabla q)_{(L^2(\Omega))^3} - (\mathcal{D}p,\mathcal{D}q)_{L^2(\Omega)}+ik(\Tr_{\Gamma_{out}}p,\Tr_{\Gamma_{out}}q)_{L^2(\Gamma_{out},\mu)} \]
\[+ \ iY \frac{Z_0}{k_0}\Big[\langle \Tr_\Gamma(\mathcal{D}_1 p),\Tr_\Gamma(\mathcal{D}_1 q)\rangle_{\chi} - K^2\langle \Tr_\Gamma p, \Tr_\Gamma q \rangle_{\chi} \Big].\]
\end{proposition}
\begin{proof}

Let $p\in H^1(\Omega)$ be a solution of \eqref{PH} and let $q\in V(\Omega)$ defined in \eqref{Space} be a test function.

Using Green's formula we find
\begin{align*}
    \int_{\Omega} \mathcal{D}^2 p\hspace{2pt} \overline{q}\hspace{2pt} d\lambda &= k_0^2\int_{\Omega} (1 - 2i\frac{M_0}{k_0}\partial_x - \frac{M_0^2}{k_0^2} \partial_x^2) p\hspace{2pt} \overline{q}\hspace{2pt} d\lambda \\
    &= k_0^2 \Biggl[ 
    \int_{\Omega} p \overline{q} d\lambda 
    - i \frac{M_0}{k_0} \langle p\cdot n_x,\Tr_{\partial \Omega} q\rangle_{B'(\partial \Omega),B(\partial \Omega)}  
    + i \frac{M_0}{k_0} \int_{\Omega} \partial_x \overline{q} \hspace{2pt} p \hspace{2pt} d \lambda \\
    &- i \frac{M_0}{k_0} \int_{\Omega} \partial_x \overline{p} \hspace{2pt} q \hspace{2pt} d \lambda
    + \frac{M_0^2}{k_0^2} \int_{\Omega} \partial_x p \hspace{2pt} \partial_x \overline{q} \hspace{2pt} d\lambda
    - \frac{M_0^2}{k_0^2} \langle \partial_x p\cdot n_x,\Tr_{\partial \Omega} q\rangle_{B'(\partial \Omega),B(\partial \Omega)}  
    \Biggr].
\end{align*}
Considering the different parts of $\partial \Omega = \Gamma \cup \Gamma_{in} \cup \Gamma_{out}$ satisfying~\eqref{EqDivisionDelOmega}, we first recall that $\Tr_{\Gamma_{in}} q = 0$. Furthermore, due to the geometry of our domain, as well as the regularity of the functions $p$ and $\partial_x p$, the analysis conducted in Remark~\ref{remarkregularity} can be applied. Finally, using that $n_x|_{\Gamma} = 0$ and $n_x|_{\Gamma_{out}} = 1$ $\mu$-a.e., we obtain the following:

\begin{align*}
     \int_{\Omega} \mathcal{D}^2 p\hspace{2pt} \overline{q}\hspace{2pt} d\lambda 
     &= \int_{\Omega} \mathcal{D} p\hspace{2pt} \overline{\mathcal{D}q}\hspace{2pt} d\lambda
    - i k_0 M_0 (\Tr_{\Gamma_{out}} p,\Tr _{\Gamma_{out}}q)_{L^2(\Gamma_{out},\mu)}
    \\&\quad- M_0^2 (\Tr_{\Gamma_{out}}(\partial_x p),\Tr_{\Gamma_{out}} q)_{L^2(\Gamma_{out},\mu)} \\
    &= (\mathcal{D}p, \mathcal{D}q)_{L^2(\Omega)} 
    + iM_0(kM_0 - k_0) (\Tr_{\Gamma_{out}} p,\Tr_{\Gamma_{out}} q)_{L^2(\Gamma_{out},\mu)}.
\end{align*}
Since $\displaystyle k_0 = \frac{w}{c_0} = \frac{w}{u_0} \frac{u_0}{c_0} = kM_0$, the equality simplifies to:
\[
\int_{\Omega} \mathcal{D}^2 p\hspace{2pt} \overline{q}\hspace{2pt} d\lambda = (\mathcal{D}p, \mathcal{D}q)_{L^2(\Omega)}.
\]
Similarly, due to the regularity of the normal derivative $\displaystyle \frac{\partial p}{\partial n}$ on the different parts of $\partial \Omega$, the analysis conducted in Remark~\ref{remarkregularity}, coupled with the generalized Green formula, yields:
\[
    \int_{\Omega} \Delta p\hspace{2pt} \overline{q}\hspace{2pt} d\lambda=-\int_{\Omega} \nabla p \nabla \overline{q}\hspace{2pt} d\lambda + \int_{\partial \Omega} \frac{\partial p}{\partial n} \Tr_{\partial \Omega}\overline{q} \hspace{2pt} d\mu.
\]
Then, decomposing the integral over $\partial \Omega = \Gamma_{out} \cup \Gamma_{in} \cup \Gamma$, and using the decomposition \eqref{squarediff},\eqref{D1andK}:

\begin{align*}
    \int_{\partial \Omega} \frac{\partial p}{\partial n} \Tr_{\partial\Omega} \overline{q}\hspace{2pt} d\mu &= 
    -ik(\Tr_{\Gamma_{out}}p,\Tr_{\Gamma_{out}}q)_{L^2(\Gamma_{out},\mu)} - iY \frac{Z_0}{k_0}
    \int_{\Gamma} \chi \Tr _\Gamma(\mathcal{D}(\mathcal{D} + iM_0\beta_v\partial_x) p) \Tr_\Gamma \overline{q} \hspace{2pt} d\mu \\ &\quad + (\eta,\Tr_\Gamma q)_{L^2(\Gamma,\mu)} \\
    &= -ik(\Tr_{\Gamma_{out}}p,\Tr_{\Gamma_{out}}q)_{L^2(\Gamma_{out},\mu)} - iY \frac{Z_0}{k_0}
    \int_{\Gamma} \chi \Tr_\Gamma((\mathcal{D}_1^2-K^2)  p) \Tr _\Gamma\overline{q} \hspace{2pt} d\mu\\ &\quad + (\eta,\Tr_\Gamma q)_{L^2(\Gamma,\mu)}.
\end{align*}
 Thus, by integration by parts on the term $\mathcal{D}_1^2$:
\begin{multline*}
\int_{\partial \Omega} \frac{\partial p}{\partial n}\Tr _{\partial \Omega}\overline{q}\hspace{2pt} d\mu = 
-ik(\Tr_{\Gamma_{out}}p,\Tr_{\Gamma_{out}}q)_{L^2(\Gamma_{out},\mu)} \\- iY \frac{Z_0}{k_0}\int_\Gamma  \chi (\Tr_\Gamma(\mathcal{D}_1 p)\Tr_\Gamma (\overline{\mathcal{D}_1 q})- K^2\Tr_\Gamma p\Tr_\Gamma \overline{q})\hspace{2pt} d\mu + (\eta,\Tr_\Gamma q)_{L^2(\Gamma,\mu)}.
\end{multline*}
Finally, using the $\langle \cdot ,\cdot \rangle_{\chi}$ notation, we get the following:
\[ \forall q\in V(\Omega), \quad (\nabla p,\nabla q)_{(L^2(\Omega))^3} - (\mathcal{D}p,\mathcal{D}q)_{L^2(\Omega)}+ik(\Tr_{\Gamma_{out}}p,\Tr_{\Gamma_{out}}q)_{L^2(\Gamma_{out},\mu)} \]
\[+ \ iY \frac{Z_0}{k_0}\Big[\langle \Tr_\Gamma(\mathcal{D}_1 p),\Tr_\Gamma(\mathcal{D}_1 q)\rangle_{\chi} - K^2\langle \Tr_\Gamma p,\Tr_\Gamma q \rangle_{\chi} \Big] = (\eta,\Tr_\Gamma q)_{L^2(\Gamma,\mu)}-(f,q)_{L^2(\Omega)},\]
which is the expected result. \end{proof}

\section{Limit Graphs of $\mathscr{B}_v$}\label{AppLimitGr}

We shall first define the notion of convergence of sets used here. We say that a family of subsets of $X$ $(A_r)_{r\in \bR}$ (indexed by $\bR$) converges to $A\subseteq X$ when $r\to \delta \in \overline{\bR}$, where $A$ is the set containing the $a\in X$ that follow the following property:
\[\exists V\in \mathcal{V}_\delta, \forall r\in V, a\in A_r. \]
Here $\mathcal{V}_\delta$ is the set of topological neighborhoods of $\delta \in \overline{\bR}$. This leads us to the following proposition:

\begin{proposition}
    Let $D_0 = D(0,1)$ and $D_1 = D(1,1)$ be the open balls of radius $1$ centered respectively around the complex numbers $0$ and $1$. We define the open half-spaces $\Im m_{> 0} = \{z\in \bC, \Im m(z) > 0\}$ and $\Im m_{< 0} = \{z\in \bC, \Im m(z) < 0\}$, and we write for simplicity $\displaystyle r = \frac{\Im m(Y)}{\Re e(Y)}$. We recall that
$\mathscr{B}_{v,r} = \{\beta\in D_0, \Re e(Y)\Re e(K^2) - \Im m(Y)\Im m(K^2) < 0\}$. Then the following holds:
    
    \noindent$(i)$ $\mathscr{B}_{v,r} \xrightarrow{r\to +\infty} D_0 \cap [(\Im m_{> 0} \cap \overline{D_1})\cup (\Im m_{< 0} \setminus D_1)].$
    
    \noindent$(ii)$ $\mathscr{B}_{v,r} \xrightarrow{r\to -\infty} D_0 \cap [(\Im m_{> 0} \setminus D_1)\cup (\Im m_{< 0} \cap \overline{D_1})].$
\end{proposition}

\begin{proof}
Let us prove $(i)$, as the proof of $(ii)$ is analogous. Let $\beta_v = x+iy\in \bC$ be in the limit set of $\mathscr{B}_{v,r}$ as $r\to +\infty$. Then $x^2 + y^2 < 1$ and there exists $R > 0$ such that for all $r \geq R$, the following condition is satisfied: 
$$\Re e(Y)\Re e(K^2) - \Im m(Y)\Im m(K^2) < 0.$$
Recalling that $\Re e(Y)$ is always positive, we ensure that $\Im m(Y) > 0$ no matter the value of $r \geq R$ chosen, 
The condition is equivalent to
$$\Re e(K^2)\cdot \frac{1}{\frac{\Im m(Y)}{\Re e(Y)}} - \Im m(K^2)   < 0.$$
Thus taking $\displaystyle \frac{\Im m(Y)}{\Re e(Y)} \longrightarrow +\infty$ we must have
$$\Im m(K^2) \geq 0.$$
Let us recall that we have
\[\Re e(K^2) = \frac{k_0^2}{4|1-\beta_v|^2}(\beta_R^2-\beta_I^2-\beta_R|\beta_v|^2) = \frac{k_0^2}{4|1-\beta_v|^2}(x^2-y^2-x(x^2+y^2))\]
and
\[\Im m(K^2) = \frac{k_0^2}{4|1-\beta_v|^2}\beta_I(2\beta_R-|\beta_v|^2) = \frac{k_0^2}{4|1-\beta_v|^2}y(2x-x^2 - y^2).\]
Let us start by assuming $\Im m(K^2) > 0$. This assumption leads to:
\[y(2x-x^2 - y^2) > 0 \]
Rewriting $2x-x^2 - y^2 = 1 - \left((x-1)^2 + y^2\right)$, we have two cases:
\begin{itemize}
    \item If $\beta_v \in \Im m_{> 0} $, then $(x-1)^2 + y^2 < 1$, meaning $\beta_v \in D(1,1)$.
    \item If $\beta_v \in \Im m_{< 0} $, then $(x-1)^2 + y^2 > 1$, meaning $\beta_v \notin \overline{D(1,1)}$.
\end{itemize}
Otherwise, we have $\Im m(K^2) = 0$. Since $\beta_v\in \mathscr{B}_{v,R}$, this implies that $\Re e(K^2) < 0$ and thus
\[x^2-y^2-x(x^2+y^2) < 0.\]
If by absurd $y = 0$, then $x^2(1-x) < 0$, which implies that $x > 1$. That is absurd, thus $y \neq 0$ and by the expression of $\Im m(K^2)$, we have that $(x-1)^2 + y^2 = 1$, or $\beta_v \in \partial D(1,1)$.

We have thus proven that $\beta_v \in D_0 \cap [(\Im m_{> 0} \cap \overline{D_1})\cup (\Im m_{< 0} \setminus D_1)]$.

Conversely, let $\beta_v \in D_0 \cap [(\Im m_{> 0} \cap \overline{D_1})\cup (\Im m_{< 0} \setminus D_1)]$. Notice that we always have $\Im m (\beta_v) \neq 0$. Let us check separately the following cases:
\begin{itemize}
    \item If $\beta_v\notin \partial D_1$, then $ \beta_v \in (\Im m_{> 0} \cap D_1)\cup (\Im m_{< 0} \setminus \overline{D_1})$. Computing $\Im m(K^2)$ as done previously, we get that $\Im m(K^2) > 0$. Setting $\displaystyle R = 2\frac{\Re e(K^2)}{\Im m(K^2)}$, we get that $ \forall r\geq R, \beta_v\in \mathscr{B}_{v,r}$.
    \item Otherwise, $\beta_v\in \partial D_1$. It implies that $\Re e(\beta_v) > 0$, as well as $\Im m(K^2) = 0$ and $\displaystyle \Re e(K^2) = -\frac{k_0^2\Re e(\beta_v)}{2|1-\beta_v|^2} < 0$. Setting $R = 1$, we get that $\forall r\geq R, \beta_v\in \mathscr{B}_{v,r}$.
\end{itemize}
In any case, $\beta_v$ is in the limit set of $\mathscr{B}_{v,r}$ as $r\to +\infty$. 
\end{proof}


\section*{Acknowledgements}
The authors thank E. Savin and F. Simon for introducing the liner modeling problems in the aircraft and F. Magoulès, O. Pironneau, and C. Bardos,  A. Teplyaev for their enthusiasm and interest in the solved
problem. The authors thank P. Sourisse for working on the physical meaning of
the generalized Myers condition and G. Claret and S. Jegou for their first considerations of this type of problem under the supervision of A. Rozanova-Pierrat. 

\bibliographystyle{siam}
\label{bib:sec}
\bibliography{ref.bib}

\end{document}